\documentclass[oneside,12pt]{amsart}
\usepackage{latexsym}
\usepackage{amsthm}
\usepackage{amsmath}
\usepackage{amsfonts}
\usepackage{amssymb}

\usepackage{graphicx}
\input xy
\xyoption{all}
\SelectTips{cm}{}

\newtheorem{theo}{Theorem}[section]
\newtheorem{lemma}[theo]{Lemma}
\newtheorem{propo}[theo]{Proposition}

\newtheorem{coro}[theo]{Corollary}

\newcommand\Ho{\operatorname{\rm Ho}}

\newcommand\colim{\operatorname{colim}}

\newcommand\id{\operatorname{id}}

\newcommand\hocolim{\operatorname{hocolim}}
\newcommand\Id{\text{Id}}
\newcommand\sSets{\text{\bf sSets}}

\newcommand\semicoloc{\operatorname{scoloc}}
\newcommand\semiloc{\operatorname{sloc}}
\newcommand\coloc{\operatorname{coloc}}
\newcommand\loc{\operatorname{loc}}
\newcommand\sk{\operatorname{sk}}

\newcommand\Ord{\operatorname{Ord}}

\newcommand\cA{\mathcal {A}}

\newcommand\cC{\mathcal {C}}
\newcommand\cD{\mathcal {D}}

\newcommand\cK{\mathcal {K}}
\newcommand\cL{\mathcal {L}}
\newcommand\cS{\mathcal {S}}
\newcommand\cT{\mathcal {T}}
\newcommand\cX{\mathcal {X}}
\newcommand\cY{\mathcal {Y}}
\newcommand\Z{\mathbb {Z}}

\newcommand{\lsemiperp}{\mbox{\tiny\text{$\reflectbox{$\mathsf{L}$}$}}}
\newcommand{\rsemiperp}{\reflectbox{$\lsemiperp$}}

\newcommand{\boldperp}{\text{\boldmath$\perp$}}

\date{\today}

\numberwithin{equation}{section}

\begin{document}
\title[Are all localizing subcategories coreflective?]{Are all localizing subcategories
of stable \\ homotopy categories coreflective?}
\author[C. Casacuberta, J. J. Guti\'{e}rrez and J. Rosick\'{y}]
{Carles Casacuberta, Javier J. Guti\'{e}rrez  and Ji\v{r}\'{\i} Rosick\'{y}$^*$}
\thanks{ $^*$ The two first-named authors were supported by the Spanish Ministry of Education
and Science under MEC-FEDER grants MTM2007-63277 and MTM2010-15831, and by the Generalitat de Catalunya
as members of the team 2009~SGR~119.
The third-named author was supported by the Ministry of Education of the Czech Republic under the project
MSM~0021622409 and by the Czech Science Foundation under grant 201/11/0528. 
He gratefully acknowledges the hospitality of the University of Barcelona
and the Centre de Recerca Matem\`atica.}
\address{\newline C. Casacuberta\newline
Institut de Matem\`atica de la Universitat de Barcelona\newline
Gran Via 585, 08007 Barcelona, Spain\newline
carles.casacuberta@ub.edu
\newline \phantom{x} \newline
J. J. Guti\'{e}rrez\newline
Centre de Recerca Matem\`atica\newline
Apartat 50\newline
08193 Bellaterra, Spain\newline
and\newline
Departament d'\`Algebra i Geometria\newline
Universitat de Barcelona\newline
Gran Via 585, 08007 Barcelona, Spain\newline
jgutierrez@crm.cat
\newline \phantom{x} \newline
J. Rosick\'{y}\newline
Department of Mathematics and Statistics\newline
Masaryk University\newline
Kotl\'{a}\v{r}sk\'{a}~2, 600~00~Brno, Czech Republic\newline
rosicky@math.muni.cz
}
\begin{abstract}
We prove that, in a triangulated category with combinatorial models,
every localizing subcategory is coreflective and every
colocalizing subcategory is reflective if a certain large-cardinal axiom 
(Vop\v{e}nka's principle) is assumed true. It follows that,
under the same assumptions, orthogonality sets up a bijective correspondence 
between localizing subcategories and colocalizing subcategories.
The existence of such a bijection was left as an open problem 
by Hovey, Palmieri and Strickland in their axiomatic study of 
stable homotopy categories and also by Neeman in the context of well-generated 
triangulated categories. 
\end{abstract}
\subjclass[2010]{18E30, 18G55, 55P42, 55P60, 03E55}

\maketitle

\begin{center}
{\sc Introduction}
\end{center}

The main purpose of this article is to address
a question asked in \cite[p.~35]{HPS} of whether every localizing
subcategory (i.e., a full triangulated subcategory closed under coproducts) 
of a stable homotopy category $\cT$ is the kernel of a localization on $\cT$
(or, equivalently, the image of a colocalization).
We prove that the answer is affirmative if $\cT$ arises from a combinatorial model
category, assuming the truth of a large-cardinal axiom from set theory called Vop\v{e}nka's principle
\cite{AR},~\cite{J}. A~model category (in the sense of Quillen) is called \emph{combinatorial}
if it is cofibrantly generated \cite{Hi},~\cite{Ho1} and its underlying category
is locally presentable \cite{AR},~\cite{GU}. Many triangulated categories of interest
admit combinatorial models, including derived categories of rings and the homotopy
category of spectra. 

More precisely, we show that, if $\cK$ is a stable combinatorial 
model category, then every semilocalizing subcategory $\cC$
of the homotopy category $\Ho(\cK)$ is coreflective under Vop\v{e}nka's principle, and the coreflection is exact
if $\cC$ is localizing. We call $\cC$ \emph{semilocalizing} if it is closed
under coproducts, cofibres and extensions, but not necessarily under fibres.
Examples include kernels of nullifications in the sense of \cite{B1} or \cite{DF} on the homotopy
category of spectra.

We also prove that, under the same hypotheses, every semilocalizing subcategory
$\cC$ is \emph{singly generated}; that is, there is an object $A$ such that $\cC$ is the 
smallest semilocalizing subcategory containing~$A$. The same result is true for localizing subcategories. 
The question of whether every localizing subcategory is singly generated in a well-generated triangulated 
category was asked in \cite[Problem~7.2]{N2}. We note that, as shown in~\cite[Proposition~6.10]{R}, 
triangulated categories with combinatorial models are well generated.

In an arbitrary triangulated category~$\cT$, localizing
subcategories need neither be singly generated nor coreflective.
Indeed, the existence of a coreflection onto a localizing subcategory $\cC$ 
is equivalent to the existence of a right adjoint for the Verdier functor $\cT\to\cT/\cC$;
see \cite[Proposition~9.1.18]{N1}.
Hence, if a coreflection onto $\cC$ exists, then $\cT/\cC$ has small hom-sets.
This need not happen if no restriction is imposed on~$\cT$;
a~counterexample was given in~\cite{CN}.

Dually, we prove that every full subcategory $\cL$ closed under
products and fibres in a triangulated category
with locally presentable models is reflective under Vop\v{e}nka's principle. 
The reflection is semiexact
if $\cL$ is closed under extensions, and it is exact
if $\cL$ is colocalizing, as in the dual case.
However, we have not been able to prove that colocalizing (or semicolocalizing) subcategories 
are necessarily singly generated, not even under large-cardinal assumptions.

This apparent lack of symmetry is not entirely surprising, in view of some well-known
facts involving torsion theories.
In abelian categories, a full subcategory closed under colimits and extensions
is called a \emph{torsion class}, and one closed under limits and
extensions is called a \emph{torsion\nobreakdash-free class}. These are analogues
of semilocalizing and semicolocalizing subcategories of triangulated categories.
Torsion theories have also been considered in triangulated categories
by Beligiannis and Reiten in~\cite{BR}, in connection with $t$\nobreakdash-structures.
In well-powered abelian categories, 
torsion classes are necessarily coreflective and torsion-free classes are reflective~\cite{Di}. 
As shown in \cite{DH1} and~\cite{GS}, Vop\v{e}nka's principle implies that every torsion
class of abelian groups is singly generated. However, 
there exist torsion\nobreakdash-free classes that are not singly generated
in~ZFC; for example, the class of abelian groups whose countable subgroups are free
\cite[Theorem~5.4]{DG}. 
(In this article, we do not make a distinction between the
terms ``singly generated'' and ``singly cogenerated''.)

Our results imply that, if $\cT$ is the homotopy category of a stable
combinatorial model category and Vop\v{e}nka's principle holds, then there
is a bijection between localizing and colocalizing subcategories of~$\cT$,
given by orthogonality. This was asked in \cite[\S~6]{S} and in~\cite[Problem 7.3]{N2}.
In fact, we prove that there is a bijection between semilocalizing and semicolocalizing
subcategories as well, and each of those determines a $t$\nobreakdash-structure
in~$\cT$.

The lack of symmetry between reflections and coreflections
also shows up in the fact that singly generated semilocalizing subcategories
are coreflective (in ZFC) in triangulated categories with combinatorial models.
A detailed proof of this claim is given in Theorem~\ref{singlygen} below;
the argument goes back to Bousfield \cite{B3}, \cite{B2} in the case of spectra, and has subsequently
been adapted to other special cases in \cite{AJS1}, \cite{BR}, \cite{HPS}, \cite{KN},~\cite{LuDAGI}
---our version generalizes some of these.
However, we do not know if singly generated semicolocalizing subcategories 
can be shown to be reflective in~ZFC. A positive answer would imply the existence of
cohomological localizations of spectra, which is so far unsettled in~ZFC.

It remains of course to decide if Vop\v{e}nka's principle
(or any other large-cardinal principle) is really needed in order to answer all these questions.
Although we cannot ascertain this, we prove that
there is a full subcategory of the homotopy category of spectra closed under
retracts and products which fails to be weakly reflective, assuming that
there are no measurable cardinals. This follows from the existence of a full subcategory
of abelian groups with the same property, and hence solves
an open problem proposed in~\cite[p.~296]{AR}. In connection with this problem, see also~\cite{Prz}.

\section{Reflections and coreflections in triangulated categories}
\label{TMC}

In this first section we recall basic concepts and fix our terminology,
which is mostly standard, except for small discrepancies in the notation
for orthogonality and localization in a number of recent
articles and monographs about triangulated categories,
such as \cite{BR}, \cite{BIK}, \cite{HPS}, \cite{Kr}, \cite{N1} or~\cite{N2}.
The essentials of triangulated categories can be found in~\cite{N1}.

For a category $\cT$, we denote by $\cT(X,Y)$ the set of morphisms 
from $X$ to~$Y$. We tacitly assume that subcategories are isomorphism-closed,
and denote indistinctly a full subcategory and the class of its objects.

\subsection{Reflections and coreflections}
\label{RAC}

A full subcategory $\cL$ of a category $\cT$ is \emph{reflective} 
if the inclusion $\cL\hookrightarrow\cT$ has a left adjoint $\cT\to\cL$.
Then the composite $L\colon\cT\to\cT$ is called a \emph{reflection} onto~$\cL$.
Such a functor $L$ will be called a \emph{localization} 
and objects in $\cL$ will be called \emph{$L$\nobreakdash-local}.
There is a natural transformation $l\colon\Id\to L$ 
(namely, the unit of the adjunction) such that $Ll\colon L\to LL$ is
an isomorphism, $lL$ is equal to~$Ll$, and, for each $X$, the morphism $l_X\colon X\to LX$ is initial
in $\cT$ among morphisms from $X$ to objects in~$\cL$.

Similarly, a full subcategory $\cC$ of $\cT$ is \emph{coreflective} if the inclusion
$\cC\hookrightarrow\cT$ has a right adjoint.
The composite $C\colon\cT\to\cT$ is called a \emph{coreflection} 
or a \emph{colocalization} onto~$\cC$,
and it is equipped with a natural transformation
$c\colon C\to\Id$ (the counit of the adjunction) such that $Cc\colon CC\to C$
is an isomorphism, $cC$ is equal to~$Cc$, and $c_X\colon CX\to X$ is terminal
in~$\cT$, for each~$X$, among morphisms from objects in $\cC$
(which are called \emph{$C$\nobreakdash-colocal}) into~$X$.

A full subcategory $\cL$ of a category $\cT$ is called
\emph{weakly reflective} if for every object $X$ of $\cT$ there is a morphism $l_X\colon X\to X^*$
with $X^*$ in $\cL$ and such that the function
\[ \cT(l_X,Y)\colon\cT(X^*,Y)\longrightarrow\cT(X,Y) \]
is surjective for all objects $Y$ of~$\cL$. Thus, every morphism from $X$ to an object of $\cL$
factors through~$l_X$, not necessarily in a unique way. If such a factorization is unique
for all objects~$X$, then the morphisms $l_X\colon X\to X^*$ for all $X$ define together a reflection,
so $\cL$ is then reflective. One defines \emph{weakly coreflective} subcategories dually.

If a weakly reflective subcategory is closed under retracts, 
then it is closed under all products that exist in~$\cT$; see~\cite[Remark~4.5(3)]{AR}.
Dually, weakly coreflective subcategories closed under retracts are
closed under coproducts. Reflective subcategories are closed under limits, 
while coreflective subcategories are closed under colimits. 

If $L$ is a reflection on an additive category~$\cT$, then the objects $X$ such that
$LX=0$ are called \emph{$L$\nobreakdash-acyclic}.
The full subcategory of $L$\nobreakdash-acyclic objects is closed under 
colimits. For a coreflection~$C$, the class of objects $X$ such that 
$CX=0$ is closed under limits, and such objects are called \emph{$C$\nobreakdash-acyclic}.

\subsection{Closure properties in triangulated categories}
\label{LACS}

 From now on, we assume that $\cT$ is a triangulated category with products
and coproducts. Motivated by topology, we denote by $\Sigma$ the shift operator and call it \emph{suspension}.
Distinguished triangles in $\cT$ will simply be called \emph{triangles}
and will be denoted by
\begin{equation}
\label{triangle}
\xymatrix{
X \ar[r]^{u} & Y \ar[r]^{v} & Z \ar[r]^{w} & \Sigma X,
}
\end{equation}
or shortly by $(u,v,w)$. We say that a functor $F\colon\cT\to\cT$
\emph{preserves} a triangle (\ref{triangle}) if $(Fu,Fv,Fw)$ is a triangle.
Note that, if this happens, then $F\Sigma X\cong \Sigma FX$.

A full subcategory $\mathcal{S}$ of~$\cT$ will be called
\begin{itemize}
\item[(i)] \emph{closed under fibres} if $X$ is in $\cS$ for every triangle~(\ref{triangle})
where $Y$ and $Z$ are in~$\cS$;
\item[(ii)] \emph{closed under cofibres} if $Z$ is in $\cS$ for every triangle~(\ref{triangle})
where $X$ and $Y$ are in~$\cS$;
\item[(iii)] \emph{closed under extensions} if $Y$ is in $\cS$ for every triangle~(\ref{triangle})
where $X$ and $Z$ are in~$\cS$;
\item[(iv)] \emph{triangulated} if it is closed under fibres, cofibres and extensions.
\end{itemize}

\medskip

A full subcategory of $\cT$ is called \emph{localizing} if it is triangulated and
closed under coproducts, and \emph{colocalizing} if it is triangulated
and closed under products.
If a triangulated subcategory $\cS$ is closed under countable coproducts or under countable products, 
then $\cS$ is automatically closed under retracts; see \cite[Lemma~1.4.9]{HPS} or~\cite[Proposition~1.6.8]{N1}.

More generally, a full subcategory of $\cT$ will be called
\emph{semilocalizing} if it is closed under coproducts, cofibres, and extensions
(hence under retracts and suspension), but not necessarily under fibres.
And a full subcategory will be called \emph{semicolocalizing} if it is closed under products, fibres, 
and extensions (therefore under retracts and desuspension as well).
Semilocalizing subcategories are also called \emph{cocomplete pre-aisles} elsewhere \cite{AJS1}, \cite{Sta}, 
and semicolocalizing subcategories are called \emph{complete pre-coaisles}
---the terms ``aisle'' and ``coaisle'' originated in~\cite{KV}.
See also \cite{BR} for a related discussion of torsion pairs and $t$\nobreakdash-structures
in triangulated categories.

A reflection $L$ on $\cT$ will be called \emph{semiexact} if the subcategory
of $L$\nobreakdash-local objects is semicolocalizing, and \emph{exact} if it is colocalizing. 
Dually, a coreflection $C$ will be called
\emph{semiexact} if the subcategory of $C$\nobreakdash-colocal objects is semilocalizing
and \emph{exact} if it is localizing.

If $L$ is a semiexact reflection with unit~$l$, then, since the class
of $L$\nobreakdash-local objects is closed under desuspension, there is a natural morphism
$\nu_X\colon LX\to\Sigma^{-1}L\Sigma X$ 
such that $\nu_X\circ l_X=\Sigma^{-1}l_{\Sigma X}$ for all~$X$, 
and hence a natural morphism
\begin{equation}
\label{natural}
\Sigma\nu_X\colon\Sigma LX\longrightarrow L\Sigma X
\end{equation}
such that $\Sigma\nu_X\circ\Sigma l_X=l_{\Sigma X}$. As we next show, if $\Sigma LX\cong L\Sigma X$
for a given object~$X$, then $\Sigma\nu_X$ is automatically an isomorphism.

\begin{lemma}
\label{sigma}
Suppose that $L$ is a semiexact reflection.
If $\Sigma LX$ is $L$\nobreakdash-local for a given object~$X$,
then $\Sigma\nu_X$ is an isomorphism.
\end{lemma}

\begin{proof}
If $\Sigma LX$ is $L$\nobreakdash-local, then there is a (unique) morphism
$h\colon L\Sigma X\to\Sigma LX$ such that $h\circ l_{\Sigma X}=\Sigma l_X$.
Thus $\Sigma\nu_X\circ h\circ l_{\Sigma X}=l_{\Sigma X}$,
which implies that $\Sigma\nu_X\circ h={\rm id}$, by the universal property
of~$L$. Similarly, $h\circ\Sigma\nu_X\circ\Sigma l_X=\Sigma l_X$,
and hence $\Sigma^{-1}h\circ\nu_X\circ l_X=l_X$, from which it follows
that $\Sigma^{-1}h\circ\nu_X={\rm id}$, or $h\circ\Sigma\nu_X={\rm id}$.
This proves that $\Sigma\nu_X$ has indeed an inverse.
\end{proof}

\begin{theo}
\label{exactlocs}
Let $\cT$ be a triangulated category. For a semiexact reflection $L$ on~$\cT$,
the following assertions are equivalent:
\begin{itemize}
\item[{\rm (i)}] $L$ is exact.
\item[{\rm (ii)}] The class of $L$\nobreakdash-local objects is closed under $\Sigma$.
\item[{\rm (iii)}] $\Sigma LX\cong L\Sigma X$ for all~$X$.
\item[{\rm (iv)}]  $\Sigma\nu_X\colon\Sigma LX\to L\Sigma X$ is an isomorphism for all~$X$.
\item[{\rm (v)}] $L$ preserves all triangles.
\end{itemize}
\end{theo}

\begin{proof}
The equivalence between (i) and (ii) follows from the definitions.
The fact that (ii) $\Rightarrow$ (iv) is given by Lemma~\ref{sigma}, and
obviously (iv) $\Rightarrow$ (iii) $\Rightarrow$ (ii).

In order to prove that (ii) $\Rightarrow$ (v), let $(u,v,w)$ be a triangle,
and let $C$ be a cofibre of~$Lu$. Thus we can choose a morphism $\varphi$ yielding a commutative diagram
of triangles
\[
\xymatrix{
X\ar[r]^-{u}\ar[d]^{l_X} & Y\ar[r]^{v}\ar[d]^{l_Y} & Z\ar[r]^{w}\ar@{.>}[d]^{\varphi} & 
\Sigma X\ar[r]^{-\Sigma u}\ar[d]^{\Sigma l_X} & \Sigma Y\ar[d]^{\Sigma l_Y} \\
LX\ar[r]_{Lu} & LY\ar[r] & C\ar[r] & \Sigma LX\ar[r]_{-\Sigma Lu} & \Sigma LY\rlap{.}
}
\]
Since $C$ is a fibre of a morphism between $L$\nobreakdash-local objects, it is
itself $L$\nobreakdash-local, since $L$ is semiexact. From the five-lemma
it follows that the morphism $\cT(C,W)\to\cT(Z,W)$ induced by $\varphi$ is an isomorphism for
every $L$\nobreakdash-local object~$W$, and therefore $\varphi$ is an $L$\nobreakdash-localization,
so $C\cong LZ$. Then the induced morphisms $LY\to LZ$ and $LZ\to L\Sigma X$ (using $\Sigma\nu_X$) are equal to
$Lv$ and $Lw$ respectively, by the universal property of~$L$. This proves that $L$ preserves
$(u,v,w)$.

Finally, (v) $\Rightarrow$ (iii), so the argument is complete.
\end{proof}

There is of course a dual result for semiexact coreflections, with a similar proof.
We omit the details.

\begin{theo}
\label{semiexactlocs}
Let $\cT$ be a triangulated category. A reflection $L$ on $\cT$
is semiexact if and only if $L$ preserves triangles $X\to Y\to Z\to \Sigma X$
where $Z$ is $L$\nobreakdash-local, and a coreflection $C$ on $\cT$ is semiexact if and only if $C$ preserves triangles 
$X\to Y\to Z\to \Sigma X$ in which $X$ is $C$\nobreakdash-colocal.
\end{theo}

\begin{proof}
We only prove the first part, as the second part is proved dually. Assume that $L$ is a semiexact reflection 
and let $C$ be a cofibre of $l_X\circ (-\Sigma^{-1}w)$. Then there is a commutative diagram of triangles
\[
\xymatrix{
\Sigma^{-1} Z\ar[r]^-{-\Sigma^{-1}w}\ar@{=}[d] & X\ar[r]^{u}\ar[d]^{l_X} & Y\ar[r]^{v}\ar@{.>}[d]^{\varphi} & Z\ar[r]^{w}\ar@{=}[d] & \Sigma X\ar[d]^{\Sigma l_X} \\
\Sigma^{-1}Z\ar[r] & LX\ar[r] & C\ar[r] & Z\ar[r] & \Sigma LX
}
\]
where $C$ is $L$\nobreakdash-local since $L$ is semiexact. Again by the five-lemma, 
$\varphi$ induces an isomorphism $\cT(C, W)\cong \cT(Y, W)$ for every $L$\nobreakdash-local object~$W$. Hence $C\cong LY$, and
the universal property of $L$ implies then that the resulting arrows $LX\to LY$,
$LY\to Z$ and $Z\to L\Sigma X$ are $Lu$, $Lv$ and $Lw$, as needed.

Conversely, let $(u,v,w)$ be a triangle where $X$ and $Z$ are $L$\nobreakdash-local. Then, since $L$ preserves this triangle, 
we have a commutative diagram
\[
\xymatrix{
\Sigma^{-1} Z\ar[r]^-{-\Sigma^{-1}w}\ar[d]^{\Sigma^{-1}l_Z} & X\ar[r]^{u}\ar[d]^{l_X} & Y\ar[r]^{v}\ar[d]^{l_Y} & Z\ar[r]^{w}\ar[d]^{l_Z} & \Sigma X\ar[d]^{\Sigma l_X\circ\Sigma \nu_X} \\
\Sigma^{-1}LZ\ar[r] & LX\ar[r]_{Lu} & LY\ar[r]_{Lv} & LZ\ar[r]_{Lw} & L\Sigma X
}
\]
where $l_X$ and $l_Z$ are isomorphisms, and $\Sigma\nu_X$ is also an isomorphism by
Lemma~\ref{sigma}. It then follows that $l_Y$ is also an isomorphism and hence 
$Y$ is $L$\nobreakdash-local. Similarly, if $Y$ and $Z$ are $L$\nobreakdash-local, then $X$ is $L$\nobreakdash-local.
Therefore, the subcategory of $L$\nobreakdash-local objects is closed under fibres and extensions, as claimed.
\end{proof}

\subsection{Orthogonality and semiorthogonality}
\label{OR}

Several kinds of orthogonality can be considered in a triangulated category.
In this article it will be convenient to use the same notation as in~\cite{BIK}.
Thus, for a class of objects $\cD$ in a triangulated category~$\cT$ with products
and coproducts, we write
\[
\begin{array}{ll}
{}^{\boldperp}\cD=\{ X \mid \mbox{$\cT(X,\Sigma^k D)=0$ for all $D\in\cD$ and $k\in\Z$} \},\\[0.3cm]
\cD^{\boldperp}=\{ Y \mid \mbox{$\cT(\Sigma^k D,Y)=0$ for all $D\in\cD$ and $k\in\Z$} \}.
\end{array}
\]

For every class of objects $\cD$, the class ${}^{\boldperp}\cD$ is localizing 
and $\cD^{\boldperp}$ is colocalizing. A~localizing subcategory $\cC$
is called \emph{closed} if $\cC={}^{\boldperp}\cD$ for some~$\cD$, or
equivalently if $\cC={}^{\boldperp}(\cC^{\boldperp})$, and a colocalizing
subcategory $\cL$ is called \emph{closed} if $\cL=\cD^{\boldperp}$ for some~$\cD$, or
equivalently if $\cL=({}^{\boldperp}\cL)^{\boldperp}$.

For example, if we work in the homotopy category of spectra and
$E$ is a spectrum, then the statement
$X\in {}^{\boldperp}E$ holds if and only if $E^*(X)=0$, where $E^* $ is the reduced cohomology theory
represented by~$E$. Thus, ${}^{\boldperp}E$ is the class of $E^*$\nobreakdash-acyclic spectra.
(Here and later, we write ${}^{\boldperp}E$ instead of ${}^{\boldperp}\{E\}$ for simplicity.)

Let us introduce the following variant, which we call \emph{semiorthogonality}:
\[
\begin{array}{ll}
{}^{\lsemiperp}\cD=\{ X \mid \mbox{$\cT(X,\Sigma^k D)=0$ for all $D\in\cD$ and $k\le 0$} \},\\[0.3cm]
\cD^{\,\rsemiperp}=\{ Y \mid \mbox{$\cT(\Sigma^k D,Y)=0$ for all $D\in\cD$ and $k\ge 0$} \}.
\end{array}
\]

Similarly as above,
for every class of objects $\cD$ the class ${}^{\lsemiperp}\cD$ is semilocalizing,
while $\cD^{\,\rsemiperp}$ is semicolocalizing. A~semilocalizing subcategory $\cC$
will be called \emph{closed} if $\cC={}^{\lsemiperp}\cD$ for some class of objects~$\cD$, or
equivalently if $\cC={}^{\lsemiperp}(\cC^{\,\rsemiperp})$. A~semicolocalizing
subcategory $\cL$ will be called \emph{closed} if $\cL=\cD^{\,\rsemiperp}$ for some~$\cD$, or
equivalently if $\cL=({}^{\lsemiperp}\cL)^{\rsemiperp}$.

Note that, if a class $\cD$ is preserved by $\Sigma$ and~$\Sigma^{-1}$, then ${}^{\lsemiperp}\cD=
{}^{\boldperp}\cD$ and $\cD^{\,\rsemiperp}=\cD^{\boldperp}$. Therefore, if a localizing
subcategory $\cC$ is closed, then it is also closed as a semilocalizing subcategory, since
$\cC={}^{\boldperp}(\cC^{\boldperp})={}^{\lsemiperp}(\cC^{\boldperp})$. The dual assertion is 
of course also true.

Semiexact reflections and semiexact coreflections are linked through
the following basic result, which generalizes Lemma~3.1.6 in~\cite{HPS}; 
compare also with \cite[Proposition~2.3]{BR}, \cite[Lemma~1.2]{BIK}, and \cite[Proposition~4.12.1]{Kr}.

\begin{theo}
\label{ref-coref}
In every triangulated category $\cT$ 
there is a bijective correspondence between semiexact reflections and
semiexact coreflections such that, if a reflection $L$ is paired
with a coreflection $C$ under this bijection, then the following hold:
\begin{itemize}
\item[{\rm (i)}] For every $X$, the morphisms $l_X\colon X\to LX$ and
$c_X\colon CX\to X$ fit into a triangle 
\[
\xymatrix{
CX \ar[r] & X \ar[r] & LX \ar[r] & \Sigma CX.
}
\]
\item[{\rm (ii)}]
The class $\cL$ of $L$\nobreakdash-local objects coincides with the
class of $C$\nobreakdash-acyclics, and the class $\cC$ of $C$\nobreakdash-colocal objects
coincides with the class of $L$\nobreakdash-acyclics. 
\item[{\rm (iii)}] The class $\cC$ is equal to ${}^{\lsemiperp}\cL$, and $\cL$ is equal to $\cC^{\,\rsemiperp}$.
\item[{\rm (iv)}] $L$ is exact if and only if $C$ is exact. In this case,
$\cC={}^{\boldperp}\cL$ and $\cL=\cC^{\boldperp}$.
\end{itemize}
\end{theo}

\begin{proof}
Let $L$ be a semiexact reflection. For every $X$ in~$\cT$, choose a fibre $CX$ of the
unit morphism $l_X\colon X\to LX$. Thus, for every $X$ in $\cT$ we have a triangle
\begin{equation}
\label{C-L_triangle}
\xymatrix{
CX\ar[r]^{c_X} & X\ar[r]^{l_X} & LX\ar[r] & \Sigma CX.
}
\end{equation}
If we apply $L$ to~(\ref{C-L_triangle}), 
since $LX$ is $L$\nobreakdash-local, Theorem~\ref{semiexactlocs} implies that 
\[
\xymatrix{
LCX\ar[r]^{Lc_X} & LX\ar[r]^{Ll_X} & LLX\ar[r] & \Sigma LCX
}
\]
is a triangle, and hence $LCX=0$. 
For each morphism $f\colon X\to Y$, choose a morphism $Cf\colon CX\to CY$ such that
the following diagram commutes:
\[
\xymatrix{
\Sigma^{-1} LX\ar[r]\ar[d]_{\Sigma^{-1}Lf} & CX\ar[r]^{c_X}\ar[d]_{Cf} & X\ar[r]^{l_X}\ar[d]^{f} & LX\ar[d]^{Lf} \\
\Sigma^{-1}LY\ar[r] & CY\ar[r]_{c_Y} & Y\ar[r]_{l_Y} & LY\rlap{.}
}
\]
Then $Cf$ is unique, since $LCX=0$ implies that $\cT(CX,\Sigma^kLY)=0$ for $k\le 0$,
and therefore $c_Y\colon CY\to Y$ induces a bijection $\cT(CX,CY)\cong\cT(CX,Y)$.
This yields functoriality of~$C$ and naturality of~$c$.
Moreover, for each~$X$ the following diagram commutes:
\[
\xymatrix{
CCX\ar[r]^{c_{CX}}\ar[d]_{Cc_X} & CX\ar[r]^{l_{CX}}\ar[d]_{c_X} & LCX\ar[d]^{Lc_X} \\
CX\ar[r]_{c_X} & X\ar[r]_{l_X} & LX\rlap{.}
}
\]
Here the fact that $LCX=0$ implies that $c_{CX}$ is an isomorphism. And, since $c_X$ induces
a bijection $\cT(CCX,CX)\cong\cT(CCX,X)$, we infer that $c_{CX}=Cc_X$ and therefore
$C$ is a coreflection.

In order to prove that $C$ is semiexact, we 
need to show that the class of $C$\nobreakdash-colocal 
objects is closed under cofibres and extensions. 
For this, let $X\rightarrow Y\rightarrow Z$ be a triangle 
and assume first that $X$ and $Z$ are
$C$\nobreakdash-colocal. Consider the following diagram, where all the columns and the central row are triangles:
\[
\xymatrix{
CX\ar[r]\ar[d] & CY\ar[r]\ar[d] & CZ\ar[d]\\
X\ar[r]\ar[d] & Y\ar[r]\ar[d] & Z\ar[d] \\
LX\ar[r] & LY\ar[r] & LZ\rlap{.}
}
\]
Since $CX\cong X$ and $CZ\cong Z$, we have that $LX=0$ and $LZ=0$. Since $L$ is semiexact, 
$L\Sigma X=0$ as well.
This implies that $\cT(Y,W)\cong\cT(Z,W)$ for every $L$\nobreakdash-local object~$W$.
Hence the morphism $LY\to LZ$ is an isomorphism. Therefore $LY=0$ and $CY\cong Y$, as needed.
Second, assume that $X$ and $Y$ are $C$\nobreakdash-colocal. Then $LX=0$ and $LY=0$ and
the same argument tells us that $LZ=0$, so $CZ\cong Z$.

Part (ii) follows directly from~(i). The first claim of (iii) is proved as follows. 
Let $X\in \cC$. Since $LX=0$ and the class of $L$\nobreakdash-local objects is closed under desuspension, we have that
\[
\cT(X,\Sigma^k D)\cong \cT(LX, \Sigma^k D)=0 \ \mbox{for $k\le 0$ and all $D\in \cL$}.
\]
This tells us that $X\in {}^{\lsemiperp}\cL$. Conversely, if $X\in {}^{\lsemiperp}\cL$, then 
$\cT(LX, LX)\cong\cT(X, LX)=0$. Therefore $LX=0$, so $X\in \cC$. The second part is proved dually.

The first claim of (iv) follows by considering the commutative diagram
\[
\xymatrix{
\Sigma CX\ar[r]^{\Sigma c_X}\ar[d] & \Sigma X\ar[r]^{\Sigma l_X}\ar@{=}[d] & \Sigma LX \ar[r]\ar[d] & \Sigma\Sigma CX\ar[d] \\
C\Sigma X\ar[r]_{c_{\Sigma X}} & \Sigma X\ar[r]_{l_{\Sigma X}} & L\Sigma X\ar[r] & \Sigma C\Sigma X\rlap{,}
}
\]
and using Theorem~\ref{exactlocs} (and its dual).
The rest is proved with the same arguments as in part~(iii).
\end{proof}

In the homotopy category of spectra, every $f$\nobreakdash-localization
functor $L_f$ in the sense of \cite{B1} or \cite{DF} is a reflection, 
and cellularizations ${\rm Cell}_A$ are coreflections. Classes of $f$\nobreakdash-local
spectra are closed under fibres, but not under cofibres nor extensions, in general. Dually,
$A$\nobreakdash-cellular classes are closed under cofibres.
Nullification functors $P_A$ (i.e., $f$\nobreakdash-localizations where $f\colon A\to 0$, 
such as Postnikov sections) are semiexact reflections. Homological localizations
of spectra (or, more generally, nullifications $P_A$ where $\Sigma A\simeq A$)
are exact reflections. One proves as in \cite{Cha}
that the kernel of a nullification $P_A$ is precisely the closure under
extensions of the image of ${\rm Cell}_A$.

\subsection{Torsion pairs and $t$-structures}
\label{CTS}

For a class of objects $\cD$ in a triangulated category $\cT$ with products
and coproducts, we denote by
$\loc(\cD)$ the smallest localizing subcategory of $\cT$ that contains~$\cD$;
that is, the intersection of all the localizing subcategories of $\cT$ that
contain~$\cD$. We use the terms $\coloc(\cD)$, $\semiloc(\cD)$, and $\semicoloc(\cD)$
analogously, and we say that each of these is \emph{generated} by~$\cD$.
If $\cD$ consists of only one object, then we say that the respective
classes are \emph{singly generated}.

Note that if $\cD=\{D_i\}_{i\in I}$ is a set (not a proper class), then
\[
\loc(\cD)=\loc\left(\textstyle\coprod_{i\in I}D_i\right) \ \mbox{and} \
\coloc(\cD)=\coloc\left(\textstyle\prod_{i\in I}D_i\right),
\]
and similarly with $\semiloc(\cD)$ and $\semicoloc(\cD)$.
Thus, in the presence of products and coproducts, ``generated by a set''
and ``singly generated'' mean the same thing.

It is important to relate classes generated by $\cD$ in this sense with the
corresponding closures of $\cD$ under orthogonality or semiorthogonality.
Although this seems to be difficult in general, 
it follows from Theorem~\ref{ref-coref}
that reflective colocalizing or semicolocalizing subcategories are closed,
and coreflective localizing or semilocalizing subcategories are also closed.
This has the following consequence.

\begin{propo}
\label{cid}
Let $\cD$ be any class of objects in a triangulated category with products and coproducts.
\begin{itemize}
\item[{\rm (i)}]
If $\semicoloc(\cD)$ is reflective, then $\semicoloc(\cD)=({}^{\lsemiperp}\cD)^{\rsemiperp}$,
and if $\semiloc(\cD)$ is coreflective, then $\semiloc(\cD)={}^{\lsemiperp}(\cD^{\,\rsemiperp})$.

\item[{\rm (ii)}]
If $\coloc(\cD)$ is reflective, then $\coloc(\cD)=({}^{\boldperp}\cD)^{\boldperp}$,
and if $\loc(\cD)$ is coreflective, then $\loc(\cD)={}^{\boldperp}(\cD^{\boldperp})$.
\end{itemize}
\end{propo}

\begin{proof}
We only prove the first claim, as the others follow similarly.
Since $\cD\subseteq \semicoloc(\cD)$, we have
$
({}^{\lsemiperp}\cD)^{\rsemiperp}\subseteq ({}^{\lsemiperp}\semicoloc(\cD))^{\rsemiperp}
$.
If $\semicoloc(\cD)$ is reflective, then it follows from part~(iii) of Theorem~\ref{ref-coref}
that $\semicoloc(\cD)$ is closed. Hence, $({}^{\lsemiperp}\cD)^{\rsemiperp}\subseteq\semicoloc(\cD)$.
The reverse inclusion follows from the fact that $({}^{\lsemiperp}\cD)^{\rsemiperp}$
is a semicolocalizing subcategory containing~$\cD$, and $\semicoloc(\cD)$ is minimal
with this property.
\end{proof}

We remark that, for every class~$\cD$, we have $\cD^{\,\rsemiperp}=\semiloc(\cD)^{\rsemiperp}$,
and similarly with left semiorthogonality or orthogonality in either side. To prove this assertion,
only the inclusion $\cD^{\,\rsemiperp}\subseteq\semiloc(\cD)^{\rsemiperp}$ needs to be checked,
and this is done as follows. Since ${}^{\lsemiperp}(\cD^{\,\rsemiperp})$
is semi\-localizing and contains~$\cD$, it also contains $\semiloc(\cD)$. Hence,
\begin{equation}
\label{closure}
\cD^{\,\rsemiperp}=({}^{\lsemiperp}(\cD^{\,\rsemiperp}))^{\rsemiperp}\subseteq\semiloc(\cD)^{\rsemiperp},
\end{equation}
as claimed.

\begin{propo}
\label{coreflecting}
Let $\cD$ be any class of objects in a triangulated category~$\cT$.
Suppose that for each $X\in\cT$ there is a triangle
$CX\to X\to LX\to \Sigma CX$ where $CX\in\semiloc(\cD)$ and $LX\in\semiloc(\cD)^{\rsemiperp}$.
Then $C$ defines a semiexact coreflection onto $\semiloc(\cD)$.
\end{propo}

\begin{proof}
If $Y$ is any object in $\semiloc(\cD)$, then $\cT(Y,LX)=0$ and $\cT(Y,\Sigma^{-1}LX)=0$.
Hence, the morphism $CX\to X$ induces a bijection $\cT(Y,CX)\cong\cT(Y,X)$, so $C$ is
a coreflection onto $\semiloc(\cD)$, hence semiexact.
\end{proof}

This fact will be used in Section~\ref{CLS}.
There is of course a dual result, and there are corresponding
facts for exact coreflections and exact reflections; cf.~\cite[Theorem~9.1.13]{N1}.

Under the hypotheses of Proposition~\ref{coreflecting}, the classes $\semiloc(\cD)$ and
$\semiloc(\cD)^{\rsemiperp}$ form a \emph{torsion pair} as defined in
\cite[I.2.1]{BR}. This yields the following fact, which also relates these notions
with $t$\nobreakdash-structures; see \cite[\S~1]{AJS1} as well.

\begin{theo}
\label{superbijection}
In every triangulated category $\cT$ there is a bijective correspondence
between the following classes:
\begin{itemize}
\item[{\rm (i)}]
Reflective semicolocalizing subcategories.
\item[{\rm (ii)}]
Coreflective semilocalizing subcategories.
\item[{\rm (iii)}]
Torsion pairs.
\item[{\rm (iv)}]
$t$-structures.
\end{itemize}
\end{theo}

\begin{proof}
The bijective correspondence between (i) and (ii) has been established
in Theorem~\ref{ref-coref}. The bijective correspondence between torsion pairs and $t$\nobreakdash-structures
is proved in \cite[I.2.13]{BR}.
If $\cC$ is a coreflective semilocalizing subcategory, then, by Proposition~\ref{coreflecting},
$(\cC,\cC^{\,\rsemiperp})$ is a torsion pair. Conversely, if $(\cX,\cY)$ is a torsion pair,
then \cite[I.2.3]{BR} tells us that $\cX$ is coreflective and semilocalizing, while $\cY$
is reflective and semicolocalizing.
\end{proof}

\subsection{Tensor triangulated categories}
\label{LOCID}

To conclude this introductory section, let $\cT$ be a \emph{tensor triangulated category},
in the sense of \cite{Ba}, \cite{HPS}, \cite{S}.
More precisely, we assume that
$\cT$ has a closed symmetric monoidal structure with a unit object~$S$, tensor product 
denoted by $\wedge$ and internal hom $F(-,-)$, compatible with
the triangulated structure and such that $\cT(X,F(Y,Z))\cong\cT(X\wedge Y,Z)$
naturally in all variables; cf.\ \cite[A.2.1]{HPS}.

Then a full subcategory $\cC$ of $\cT$ is called an
\emph{ideal} if $E\wedge X$ is in~$\cC$ for every $X$ in~$\cC$ 
and all $E$ in~$\cT$, and a full subcategory
$\cL$ is called a \emph{coideal} if $F(E,X)$ is in~$\cL$
for every $X$ in~$\cL$ and all $E$ in~$\cT$. 
A~\emph{localizing ideal} is a localizing subcategory that is also an ideal,
and similarly in the dual case.

In the homotopy category of spectra, all localizing subcategories
are ideals and all colocalizing subcategories are
coideals. As shown in \cite[Lemma 1.4.6]{HPS}, the same happens
in any \emph{monogenic} stable homotopy category (i.e., such that
the unit of the monoidal structure is a small generator).
In \cite{HPS}, for stable homotopy categories, the terms \emph{localization}
and \emph{colocalization} were used in a more restrictive sense
than in the present article. Thus an exact reflection $L$ was called 
a localization in \cite[Definition 3.1.1]{HPS}
if $LX=0$ for an object $X$ implies that $L(E\wedge X)=0$ for every~$E$;
in other words, if the class of $L$\nobreakdash-acyclic objects is a localizing ideal.
Dually, an exact coreflection $C$ was called a colocalization if
$CX=0$ implies that $C(F(E,X))=0$ for every~$E$, i.e., if the class
of $C$\nobreakdash-acyclic objects is a colocalizing coideal.

For each class of objects $\cD$, the class of those $X$ such that $F(X,D)=0$ for all $D\in\cD$
is a localizing ideal, and the class of those $Y$ such that $F(D,Y)=0$ for all $D\in\cD$
is a colocalizing coideal.
If $\cC$ is a localizing ideal, then $\cC^{\boldperp}$ is a colocalizing coideal,
and, if $\cL$ is a colocalizing coideal, then ${}^{\boldperp}\cL$ is a localizing ideal.
Thus, under the bijective correspondence given by Theorem~\ref{ref-coref},
localizations in the sense of \cite{HPS} are also paired with colocalizations.

\section{Reflective colocalizing subcategories}
\label{RCS}

Background on locally presentable and accessible categories can be found in~\cite{AR} and~\cite{MP}.
The basic definitions are as follows. Let $\lambda$ be a regular cardinal.
A~nonempty small category is \emph{$\lambda$\nobreakdash-filtered}
if, given any set of objects $\{A_i\mid i\in I\}$ where $|I|<\lambda$, there is an object $A$ and a
morphism $A_i\to A$ for each $i\in I$, and, moreover, given any set of parallel arrows
between any two objects $\{\varphi_j\colon B\to C \mid j\in J\}$ where $|J|<\lambda$, there is a morphism
$\psi\colon C\to D$ such that $\psi\circ\varphi_j$ is the same morphism for all $j\in J$.

Let $\cK$ be any category. A diagram $D\colon I\to\cK$ where $I$ is a $\lambda$\nobreakdash-filtered
small category is called a \emph{$\lambda$\nobreakdash-filtered diagram}, and, if $D$ has a colimit, then
$\colim_I D$ is called a
\emph{$\lambda$\nobreakdash-filtered colimit}. An object $X$ of $\cK$ is \emph{$\lambda$\nobreakdash-presentable}
if the functor $\cK(X,-)$ preserves $\lambda$\nobreakdash-filtered colimits.
The category $\cK$ is \emph{$\lambda$\nobreakdash-accessible}
if all $\lambda$\nobreakdash-filtered colimits exist in~$\cK$
and there is a set $\cS$ of $\lambda$\nobreakdash-presentable objects such that every object of $\cK$ is a 
$\lambda$\nobreakdash-filtered colimit of objects from~$\cS$. It is called
\emph{accessible} if it is $\lambda$\nobreakdash-accessible for some~$\lambda$.
A cocomplete accessible category is called \emph{locally presentable}.

Thus, the passage from locally presentable to 
accessible categories amounts to weakening the assumption
of cocompleteness by imposing only that enough colimits exist.
As explained in \cite[\S~2.1]{AR}, using directed colimits instead of filtered
colimits in the definitions leads to the same concepts of accessibility and
local presentability.

As explained in \cite[6.3]{AR}, \emph{Vop\v{e}nka's principle} 
is equivalent to the statement
that, \emph{given any family of objects $X_s$ of an accessible category indexed
by the class of all ordinals, there is a morphism $X_s\to X_t$ for some ordinals $s<t$.}

A functor $\gamma\colon\cK\to\cT$ between two categories 
will be called \emph{essentially surjective on sources}
if, for every object $X$ and every collection of morphisms $\{f_i\colon X\to X_i\mid i\in I\}$ in~$\cT$ 
(where $I$ is any discrete category, possibly a proper class),
there is an object $K$ and a collection of morphisms $\{g_i\colon K\to K_i\mid i\in I\}$ in~$\cK$ together
with isomorphisms $h\colon\gamma K\to X$ and $h_i\colon\gamma K_i\to X_i$ for all $i$
rendering the following diagram commutative:
\begin{equation}
\label{sources}
\xymatrix{
\gamma K \ar[d]_{h}^{\cong} \ar[r]^{\gamma g_i} & \gamma K_i \ar[d]^{h_i}_{\cong} \\
X \ar[r]_{f_i} & X_i \rlap{.}
}
\end{equation}
For example, if $\cK$ is a model category \cite{Ho1} and
$\gamma\colon\cK\to\Ho(\cK)$ is the canonical functor onto the corresponding homotopy
category (which we assume to be the identity on objects), 
then $\gamma$ is essentially surjective on sources ---and also on sinks;
here ``sources'' and ``sinks'' are meant as in~\cite{AHS}. Indeed, 
given a collection $\{f_i\colon X\to X_i\mid i\in I\}$ in~$\Ho(\cK)$, we can choose
a cofibrant replacement $q\colon K\to X$ and a fibrant replacement $r_i\colon X_i\to K_i$
for each~$i$. Then we can pick a morphism $g_i\colon K\to K_i$ for each~$i$ such that
the zig-zag
\[
\xymatrix{
X & K\ar[l]_q\ar[r]^{g_i} & K_i & X_i\ar[l]_{r_i}
}
\]
represents $f_i$. Then the choices $h=\gamma q$ and $h_i=(\gamma r_i)^{-1}$ render \eqref{sources} commutative.

\begin{theo}
\label{thm1}
Let $\cT$ be a category with products and suppose given a
functor $\gamma\colon\cK\to\cT$ where $\cK$ is accessible and $\gamma$ is essentially surjective on sources.
If Vop\v{e}nka's principle holds, then every full subcategory \mbox{$\cL\subseteq\cT$}
closed under products is weakly reflective.
\end{theo}

\begin{proof}
Write $\cL$ as the union of an ascending chain of full subcategories
indexed by the ordinals,
\[ \cL=\bigcup_{i\in\Ord} \cL_i, \]
where each $\cL_i$ is the closure under products of a small subcategory~$\cA_i$.
For each object $X$ of~$\cT$, let $X_i$ be the product of
the codomains of all morphisms from $X$ to objects of~$\cA_i$, and let
$f_i\colon X\to X_i$ be the induced morphism. Then 
every morphism from $X$ to some object of $\cL_i$ factors through $f_i$ and
hence $f_i$ is a weak reflection of $X$ onto~$\cL_i$.

Now, as in \cite[6.26]{AR}, in order to prove that $\cL$
is weakly reflective, it suffices to find an ordinal $i$ such
that, for all $j\ge i$, the morphism $f_j$ can be factorized as $f_j={\varphi}_{ij}\circ f_i$
for some $\varphi_{ij}\colon X_i\to X_j$. In other words,
\[ (X\downarrow\cL)(f_i,f_j)\ne \emptyset \]
for all $j\ge i$, where $(X\downarrow\cL)$ denotes the comma category of $\cL$ under~$X$.
Suppose the contrary. Then there are ordinals
$i_0 < i_1 < \cdots < i_s < \cdots$,
where $s$ ranges over all the ordinals, such that
\begin{equation}
\label{nomaps}
(X\downarrow\cL)(f_{i_s},f_{i_t})=\emptyset
\end{equation}
if $s < t$.
Since $\gamma$ is essentially surjective on sources,
there is a morphism $g_i\colon K\to K_i$ in $\cK$ for each ordinal~$i$,
and there are isomorphisms $h\colon \gamma K\to X$ and $h_i\colon\gamma K_i\to X_i$
such that $f_i\circ h=h_i\circ\gamma g_i$ for all~$i$.
Then $(K\downarrow\cK)(g_{i_s},g_{i_t})=\emptyset$
if $s < t$, since, if there is a morphism $G\colon K_{i_s}\to K_{i_t}$ with $G\circ g_{i_s}=g_{i_t}$,
then $F=h_{i_t}\circ\gamma G\circ (h_{i_s})^{-1}$ satisfies $F\circ f_{i_s}=f_{i_t}$,
contradicting~\eqref{nomaps}.
Since the category $(K\downarrow\cK)$ is accessible by~\cite[2.44]{AR},
this is incompatible with Vop\v{e}nka's principle, according to~\cite[6.3]{AR}.
\end{proof}

We next show that the existence of a weak reflection implies the
existence of a reflection under assumptions that do not require
any further input from large-cardinal theory.
We say that \emph{idempotents split} in a category $\cT$ if
for every morphism $e\colon A\to A$ such that $e\circ e=e$
there are morphisms $f\colon A\to B$ and $g\colon B\to A$ such that
$e=g\circ f$ and $f\circ g=\id$. This is automatic in a category
with coequalizers, since $f$ can be chosen to be a coequalizer
of $e$ and the identity, and $g$ is determined by the universal property
of the coequalizer. It also holds in other important cases;
for instance, by \cite[Proposition 1.6.8]{N1}, idempotents
split in any triangulated category with countable coproducts
or countable products.

We will need the following result, which, 
as pointed out to us by Chorny, can be derived from \cite[VII.28H]{HS}.
A~self-contained proof is given here for the sake of completeness.

Recall that a \emph{weak limit} of a diagram $D\colon I\to \cT$,
where $\cT$ is any category and $I$ is a small category, 
is an object $X$ of $\cT$ together with a natural transformation
$\nu\colon X\to D$ (where $X$ is seen as a constant functor, so $\nu$ is a
cone to~$D$) such that any other natural
transformation $Y\to D$ with $Y$ in $\cT$ factorizes through~$\nu$,
not necessarily in a unique way. Weak colimits are defined dually.

\begin{theo}\label{thm2}
Let $\cT$ be a category with products where idempotents split.
Let $\cL$ be a weakly reflective subcategory of $\cT$
closed under retracts, and assume that
every pair of parallel arrows in $\cL$ has a weak equalizer
that lies in~$\cL$. Then $\cL$ is reflective.
\end{theo}

\begin{proof}
Note first that, by \cite[Remark~4.5(3)]{AR}, since $\cL$ is weakly reflective
and closed under retracts, it is also closed under products. Recall also that
reflective or weakly reflective subcategories are tacitly assumed to be full. 

Let $A$ be any object of $\cT$ and let $r_0\colon A\to A_0$ be
a weak reflection of $A$ onto~$\cL$.
Let $I$ denote the set of all pairs of morphisms $(f,g)\colon A_0\rightrightarrows A_0$
such that $f\circ r_0=g\circ r_0$, and let
$u_1\colon A_1\to A_0$ be a weak equalizer of the pair
$(\prod_{i\in I}f_i,\,\prod_{i\in I}g_i)\colon A_0\rightrightarrows\prod_{i\in I}A_0$.
By hypothesis, we may choose $u_1$ in~$\cL$.

Since $u_1$ is a weak equalizer, there is a morphism $r_1\colon A\to A_1$ such
that $u_1\circ r_1=r_0$. Moreover, since $r_0$ is a weak reflection
and $A_1$ is in $\cL$, there is a morphism $t_1\colon A_0\to A_1$ such that
$t_1\circ r_0=r_1$. Then $(u_1\circ t_1,\,{\rm id})\in I$ and hence
$u_1\circ t_1\circ u_1 = u_1$.
It~follows that $t_1\circ u_1$ is idempotent and hence it splits.
That is, there are morphisms
$u_2\colon A_2\to A_1$ and $t_2\colon A_1\to A_2$
such that $u_2\circ t_2= t_1\circ u_1$ and $t_2\circ u_2=\id$.

We next prove that, if we pick $r_2=t_2\circ r_1$, then $r_2$ is a reflection of $A$ onto~$\cL$.
First of all, $A_2$ is a retract of $A_1$ and hence $A_2$ is in $\cL$.
Second, from the equality $r_0=u_1\circ u_2\circ r_2$ it follows that
$r_2$ is a weak reflection of $A$ onto~$\cL$.
Now, given a morphism $f\colon A\to X$ with $X$ in $\cL$, since $r_2$
is a weak reflection, there is a morphism $g\colon A_2\to X$
such that $g\circ r_2=f$. Suppose that there is another
$h\colon A_2\to X$ with $h\circ r_2=g\circ r_2$. Let $w\colon B\to A_2$
be a weak equalizer of $g$ and $h$ with $B$ in $\cL$. Then, as we next
show, $w$ has a right inverse, so the equality $g\circ w=h\circ w$
implies that $g=h$, as needed.

In order to prove that $w$ has a right inverse, note that, since
$h\circ r_2=g\circ r_2$, there is a morphism $t\colon A \to B$
with $w\circ t=r_2$. Since $r_0$
is a weak reflection of $A$ onto~$\cL$, there is a morphism $s\colon A_0\to B$
such that $s\circ r_0=t$.
Now $u_1\circ u_2 \circ w \circ s \circ r_0=r_0$; hence $(u_1\circ u_2\circ w\circ s,\,{\rm id})\in I$,
from which it follows that $u_1\circ u_2 \circ w \circ s \circ u_1=u_1$.
Finally, note that $t_2\circ t_1\circ u_1\circ u_2=t_2\circ u_2\circ t_2\circ u_2={\rm id}$
and therefore $w\circ s\circ u_1=t_2\circ t_1\circ u_1=t_2$, so $w\circ s\circ u_1\circ u_2={\rm id}$,
as claimed.
\end{proof}

\begin{coro}
\label{refinement}
Every weakly reflective subcategory closed under retracts and fibres in a triangulated
category with products is reflective.
\end{coro}

\begin{proof}
This is implied by Theorem~\ref{thm2}, since a fibre of $f-g$ is
a weak equalizer of two given parallel arrows $f$ and~$g$,
and idempotents split in a triangulated category if countable products exist,
according to \cite[Remark~1.6.9]{N1}.
\end{proof}

Dually, every weakly coreflective subcategory closed under retracts and cofibres
in a triangulated category $\cT$ with coproducts is coreflective. Neeman
proved this fact in \cite[Proposition~1.4]{N3} for thick subcategories,
without assuming the existence of coproducts in~$\cT$,
but imposing that idempotents split in~$\cT$.

Putting together Theorem~\ref{thm1} and Theorem~\ref{thm2}, we state the main result of this section.
A model category $\cK$ is called \emph{stable} \cite[2.1.1]{SS} if it is pointed (i.e., the
unique map from the initial object to the terminal object is an isomorphism)
and the suspension and loop operators are inverse equivalences on the
homotopy category $\Ho(\cK)$. It then follows that $\Ho(\cK)$ is
triangulated, where the triangles come from fibre or cofibre sequences in~$\cK$ (see \cite[6.2.6]{Ho1}), 
and has products and coproducts over arbitrary index sets, coming
from those of~$\cK$. In fact, $\Ho(\cK)$ has weak limits and
colimits, but it is neither complete nor cocomplete in general; see~\cite[2.2]{HPS}.
Quillen equivalences of stable model categories preserve fibre and cofibre sequences and hence
the triangulated structure of~$\Ho(\cK)$.

\begin{theo}
\label{mainthm1}
Let $\cK$ be a locally presentable category with a stable model category structure.
If Vop\v{e}nka's principle holds, then every full subcategory $\cL$ of $\Ho(\cK)$ closed under 
products and fibres is reflective. 
If $\cL$ is semicolocalizing, then the reflection is semiexact. 
If $\cL$ is colocalizing, then the reflection is exact.
\end{theo}

\begin{proof}
The canonical functor $\gamma\colon \cK\to \Ho(\cK)$ is
essentially surjective on sources and sinks. Hence $\gamma$ satisfies the assumptions of Theorem~\ref{thm1},
from which it follows that $\cL$ is weakly reflective. 
Closure of $\cL$ under retracts follows from the Eilenberg swindle, as in \cite[Lemma 1.4.9]{HPS}.
Then Corollary~\ref{refinement} implies that $\cL$ is in fact reflective.
The other statements hold by the definitions of the terms involved.
\end{proof}

We do not know if the assumption that $\cL$ be closed under fibres is necessary
for the validity of Theorem~\ref{mainthm1}. We note however that an important kind
of reflective subcategories, namely classes of $f$\nobreakdash-local objects
in the sense of \cite{DF} or \cite{Hi} in homotopy categories of suitable model categories, 
are closed under fibres. 

Under the assumptions of Theorem~\ref{mainthm1},
if $\Ho(\cK)$ is tensor triangulated and
a given colocalizing subcategory $\cL\subseteq\Ho(\cK)$ is a coideal,
then the exact reflection $L$
given by Theorem~\ref{mainthm1} is automatically a localization in the sense of~\cite{HPS}, 
that is, the class of $L$\nobreakdash-acyclic objects is then a localizing ideal.

\begin{coro}
\label{closed_implies_coreflective}
Let $\cK$ be a locally presentable stable model category.
If Vop\v{e}nka's principle holds, then every closed semilocalizing subcategory 
of $\Ho(\cK)$ is coreflective.
\end{coro}

\begin{proof}
Let $\cC$ be a closed semilocalizing subcategory of $\Ho(\cK)$. Then $\cC^{\,\rsemiperp}$ is a semicolocalizing
subcategory, which is reflective by Theorem~\ref{mainthm1}. Hence ${}^{\lsemiperp}(\cC^{\,\rsemiperp})$
is coreflective by Theorem~\ref{ref-coref}, and it is equal to $\cC$ since $\cC$ is closed by assumption.
\end{proof}

As observed in Subsection~\ref{OR}, if a localizing subcategory $\cC$ is closed, then it is also closed
if viewed as a semilocalizing subcategory. 
Hence, the statement of Corollary~\ref{closed_implies_coreflective} is also true for
closed localizing subcategories.

It would be very interesting to have a counterexample (if there is one) to the statement of Theorem~\ref{mainthm1}
under some set-theoretical assumption incompatible with Vop\v{e}nka's principle. We next give
a partial result in this direction, based on \cite{CSS} and~\cite{DG}, which shows that
Theorem~\ref{thm1} cannot be proved in~ZFC.
This result answers the second part of Open Problem~5 from~\cite[p.~296]{AR}.

\begin{propo}
\label{nomeasurables}
Assuming the nonexistence of measurable cardinals, there is a full subcategory
of the category of abelian groups which is closed under products and
retracts but not weakly reflective.
\end{propo}

\begin{proof}
Let $\cC$ be the closure of the class of groups $\Z^\kappa/\Z^{<\kappa}$
under products and retracts, where $\kappa$ runs over all cardinals,
and $\Z^{\kappa}$ denotes a product of copies of the integers indexed by $\kappa$
while $\Z^{<\kappa}$ denotes the subgroup of sequences whose support (i.e., the set of
nonzero entries) has cardinality smaller than~$\kappa$.
Assume that $w\colon\Z\to A$ is a weak
reflection of $\Z$ onto~$\cC$. Then there is a retraction
\[ \prod_{i\in I} \Z^{\kappa_i}/\Z^{<\kappa_i}\stackrel{r}{\longrightarrow} A \]
for some set of cardinals $\{\kappa_i\}_{i\in I}$. 
Choose a regular cardinal $\lambda$ bigger than the sum $\Sigma_{i\in I}\,\kappa_i$.
Let $d\colon\Z\to\Z^{\lambda}$ be the diagonal and
$p\colon\Z^\lambda\to\Z^\lambda/\Z^{<\lambda}$ the projection.
Since $w$ is a weak reflection, there is a homomorphism $f\colon A\to\Z^\lambda/\Z^{<\lambda}$ with
$f\circ w=p\circ d$. Following \cite[Lemma~6.1]{CSS}, there is a homomorphism $g\colon A\to\Z^\lambda$ such
that $f=p\circ g$. Since the image of~$d$ is not contained in $\Z^{<\lambda}$, 
we have $f\neq 0$ and thus $g\neq 0$. 
Since $r$ is an epimorphism, $g\circ r\ne 0$.
Now, since $\Z^{\lambda}$ maps onto $\prod_{i\in I}\Z^{\kappa_i}$,
there is a nonzero homomorphism $h\colon\Z^{\lambda}\to \Z^{\lambda}$ which vanishes
on the direct sum $\oplus_{i<\lambda}\,\Z$, since it factors through $\prod_{i\in I}\Z^{\kappa_i}/\Z^{<\kappa_i}$.
Hence, by composing $h$ with a suitable projection, we obtain 
a nonzero homomorphism $\Z^{\lambda}\to \Z$ which vanishes
on the direct sum $\oplus_{i<\lambda}\,\Z$.
According to \cite[94.4]{F}, this fact implies the existence of measurable cardinals.
This contradiction proves the statement.
\end{proof}

\begin{coro}
Assuming the nonexistence of measurable cardinals, there is a full subcategory
of the homotopy category of spectra which is closed under products
and retracts but not weakly reflective.
\end{coro}

\begin{proof}
Consider the full embedding $H$ of the category of abelian groups into the
homotopy category of spectra given by assigning to each abelian group $A$
an Eilenberg--Mac Lane spectrum $HA$ representing ordinary cohomology with coefficients in~$A$.
Since $H$ preserves products and its image is closed under retracts, it sends the class
$\cC$ considered in the proof of Proposition~\ref{nomeasurables}
to a class $H\cC$ of spectra closed under products and retracts.
This class $H\cC$ is not weakly reflective, since the
above argument shows that $H\Z$ does not admit a weak reflection onto $H\cC$
if there are no measurable cardinals.
\end{proof}

\section{Coreflective localizing subcategories}
\label{CLS}

In Section~\ref{RCS} we proved that, if $\cK$ is a locally presentable category with
a stable model category structure, then Vop\v{e}nka's principle implies
that all colocalizing subcategories (in fact, all full subcategories closed under products and fibres) 
of $\Ho(\cK)$ are reflective.
Our purpose in this section is to study if Vop\v{e}nka's principle also implies that all 
localizing or semilocalizing subcategories of $\Ho(\cK)$ are coreflective.
By Corollary~\ref{closed_implies_coreflective}, this is equivalent to asking
if they are closed.

We provide an affirmative answer by assuming that $\cK$ be
cofibrantly generated~\cite[Definition~2.1.17]{Ho1}, in addition to being locally presentable. 
Recall that a cofibrantly generated model category whose
underlying category is locally presentable is called \emph{combinatorial}.
It was shown in \cite{Dug2} and \cite{Dug1} that a model category is combinatorial if and only if
it is Quillen equivalent to a localization of some
category of diagrams of simplicial sets with respect to a set of morphisms.
Hence, examples abound. When $\cK$ is combinatorial, 
we not only prove that semilocalizing subcategories of $\Ho(\cK)$ are coreflective,
but we moreover show that they are singly generated.

Although, for the validity of our arguments, we need that $\cK$ be \emph{simplicial} \cite[II.3]{GJ},
it is not necessary to impose this as a restriction, due to the following fact.

\begin{propo}
\label{Dugger}
Every (stable) combinatorial model category is Quillen
equivalent to a (stable) simplicial combinatorial model category.
\end{propo}

\begin{proof}
For a small category $\cC$, denote by $U_+\cC$, as in \cite{Dug3}, the category
of functors from $\cC^{\rm op}$ to the category $\sSets_*$ of pointed simplicial sets.
According to \cite[Corollary~6.4]{Dug2} and \cite[Proposition~5.2]{Dug3}, 
for every combinatorial model category $\cK$
there is a small category $\cC$ such that $\cK$ is Quillen equivalent to the 
left Bousfield localization of $U_+\cC$ with respect to a certain set of morphisms. 
The category $U_+\cC$ is combinatorial, pointed, simplicial and left proper, 
and so is any of its localizations.
Since Quillen equivalences preserve the suspension and loop functors, every pointed model category
which is Quillen equivalent to a stable one is itself stable.
\end{proof}

One crucial property of combinatorial model categories that we will use
in this section is the following. For every combinatorial model category $\cK$
there is a regular cardinal $\lambda$ such that, if $X\colon I\to\cK$
and $Y\colon I\to\cK$ are diagrams where $I$ is a small $\lambda$\nobreakdash-filtered category,
and a morphism of diagrams $f\colon X\to Y$ is given such that $f_i\colon X_i\to Y_i$
is a~weak equivalence for each $i\in I$, then the induced map 
$\colim_I X\to\colim_I Y$ is also~a weak equivalence. For a proof of this fact,
see \cite[Proposition~2.3]{Dug1}. 

Another feature of combinatorial model categories is that, if $\cK$ is combinatorial
and $I$ is any small category, then the \emph{projective model structure}
(in which weak equivalences and fibrations are objectwise) and the \emph{injective model structure}
(in which weak equivalences and cofibrations are objectwise)
exist on the diagram category~$\cK^I$; see~\cite[Proposition~A.2.8.2]{Lu}.
In fact, as shown in \cite[Theorem~11.6.1]{Hi}, 
for the existence of the projective model structure it is enough that
$\cK$ be cofibrantly generated.

If $\cK^I$ is equipped with the projective model structure,
then the constant functor $\cK\to\cK^I$ is right Quillen and therefore
its left adjoint $\colim_I\colon\cK^I\to\cK$ is left Quillen, so it preserves
cofibrations, trivial cofibrations, and weak equivalences between
cofibrant diagrams \cite[II.8.9]{GJ}. Hence, its total left derived functor
$\hocolim_I$ exists.

Since we will need to use explicit formulas to compute homotopy colimits,
we recall, before going further, 
a number of basic facts about homotopy colimits in model categories.
Our main sources are \cite{BK}, \cite{Ga}, \cite{GJ}, \cite{Hi}, \cite{HV}, \cite{Ho1}, \cite{Sh}.
For simplicity, we restrict our discussion to pointed simplicial model categories,
which is sufficient for our purposes. The unpointed case would be treated analogously.

\subsection{A review of homotopy colimits}

Let $\cK$ be a pointed simplicial model category. Let $*$
be the initial and terminal object, and let
$\otimes$ denote the tensoring of $\cK$ over \emph{pointed} simplicial sets.
For each simplicial set~$W$, we denote by $W_+$ its union with a disjoint basepoint.

Let $\mathbf\Delta$ denote the category whose objects are finite ordered sets $[n]=(0,1,\dots,n)$
for $n\ge 0$, and whose morphisms are nondecreasing functions.
Let $\Delta[n]$ be the simplicial set whose set of $k$\nobreakdash-simplices
is the set of morphisms $[k]\to[n]$ in $\mathbf\Delta$,
and denote by $\Delta_+\colon {\mathbf\Delta}\to\sSets_*$ the functor 
that sends $[n]$ to~$\Delta[n]_+$.
If $X$ is a cofibrant object in~$\cK$,
then $X\otimes\partial\Delta[1]_+\to X\otimes\Delta[1]_+$ is a cofibration yielding
a cylinder for~$X$, and hence $\Sigma X\simeq X\otimes S^1$; cf.~\cite[6.1.1]{Ho1}

The \emph{realization} $|B|$ of a simplicial object $B\colon{\mathbf\Delta}^{\rm op}\to \cK$
is the coequalizer of the two morphisms
\begin{equation}
\label{realization}
\coprod_{[m]\to [n]} B_n\otimes\Delta[m]_+
\xymatrix{ {} \ar@<+0.7ex>[r] \ar@<-0.7ex>[r] & {} }
\coprod_{[n]} B_n\otimes\Delta[n]_+
\end{equation}
induced by $B_n\to B_m$ and $\Delta[m]_+\to \Delta[n]_+$, 
respectively, for each morphism $[m]\to [n]$ in~$\mathbf\Delta$;
see~\cite[VII.3.1]{GJ}.
Using coend notation \cite[IX.6]{Mac}, this can be written as
\[
|B|= \int^n B_n\otimes \Delta[n]_+ = B\otimes_{{\mathbf\Delta}^{\rm op}}\Delta_+.
\] 

Suppose given functors $X\colon I\to\cK$ and $W\colon I^{\rm op}\to\sSets_*$,
where $W$ will be called a \emph{weight}.
The (two-sided) \emph{bar construction} $B(W,I,X)\in\cK^{{\mathbf\Delta}^{\rm op}}$ 
is the simplicial object with
\begin{equation}
\label{bar}
B(W,I,X)_n=\coprod_{i_n\to\,\cdots\,\to i_0} X_{i_n}\otimes W_{i_0},
\end{equation}
whose $k$th face map omits $i_k$ using the identity on $X_{i_n}\otimes W_{i_0}$
if $0<k<n$, and using $W_{i_0}\to W_{i_1}$ if $k=0$ and
$X_{i_n}\to X_{i_{n-1}}$ if $k=n$.
Degeneracies are given by insertions of the identity.
If we choose as weight the constant diagram $S$ at the $0$th sphere~$S^0$,
then we denote $B_IX=B(S,I,X)$ and call it
a \emph{simplicial replacement} of~$X$.

The (pointed) \emph{homotopy colimit} of a functor $X\colon I\to\cK$ is defined as
\begin{equation}
\label{defhocolim}
\hocolim_I X=|B_IX|.
\end{equation}

It follows that homotopy colimits commute; that is, given $X\colon I\times J\to\cK$, 
\[
\hocolim_I\,\hocolim_JX\cong\hocolim_{I\times J}X\cong \hocolim_J\,\hocolim_I X.
\]

 From (\ref{defhocolim}) and (\ref{realization}) one obtains the \emph{Bousfield--Kan formula}
\cite[XII.2.1]{BK}, \cite[18.1.2]{Hi}, as follows.
Let $N(i\downarrow I)^{\rm op}$ be the nerve of the category $(i\downarrow I)^{\rm op}$ for each $i\in I$.
Thus, $N(i\downarrow I)^{\rm op}_+$ is the realization of the 
simplicial space ${\mathbf\Delta}^{\rm op}\to\sSets_*$ that consists in degree~$n$ of a
coproduct of copies of $S^0$ indexed by the set of sequences $i\to i_n\to\cdots\to i_0$ 
of morphisms in~$I$. Since coequalizers commute, we have
\[
\hocolim_I X\cong {\rm coeq}
\left[
\coprod_{i\to j} X_i\otimes N(j\downarrow I)^{\rm op}_+
\xymatrix{ {} \ar@<+0.7ex>[r] \ar@<-0.7ex>[r] & {} }
\coprod_i X_i\otimes N(i\downarrow I)^{\rm op}_+
\right]
= X\otimes_I N(-\downarrow I)^{\rm op}_+.
\]
(We note that $(I\downarrow i)$ was used in \cite{BK} instead of~$(i\downarrow I)^{\rm op}$.)
In other words, $\hocolim_I X$ is a weighted colimit of $X$ with weight
$N(-\downarrow I)^{\rm op}_+\colon I^{\rm op}\to\sSets_*$.

For simplicial objects ${\mathbf\Delta}^{\rm op}\to\cK$, the \emph{$n$th~skeleton}
$\sk_n$ is the composite of the truncation functor
$\cK^{{\mathbf\Delta}^{\rm op}}\to\cK^{{\mathbf\Delta}_n^{\rm op}}$ with its left adjoint,
where ${\mathbf\Delta}_n$ is the full subcategory of $\mathbf\Delta$ with objects
$\{[0],\dots,[n]\}$; see~\cite[VII.1.3]{GJ}. Thus $(\sk_n B)_m\cong B_m$ if $m\le n$, and hence 
\[
B\cong\colim_n\sk_n B
\] 
for every simplicial object~$B$.

The $n$th \emph{latching object} of $B\colon{\mathbf\Delta}^{\rm op}\to\cK$
is defined as $L_nB=(\sk_{n-1}B)_n$ for each~$n$. As explained in \cite[VII.2]{GJ},
the category $\cK^{{\mathbf\Delta}^{\rm op}}$ of simplicial objects in $\cK$ admits a model structure
(called \emph{Reedy model structure}) where weak equivalences are objectwise and cofibrations are morphisms
$f\colon X\to Y$ such that $L_nY\coprod_{L_nX}X_n\to Y_n$ is a cofibration for all~$n$.
Thus an object $B$ is Reedy cofibrant if and only if the natural morphisms $L_nB\to B_n$
are cofibrations in $\cK$ for all~$n$.

If $B$ is Reedy cofibrant, then
each skeleton $\sk_n B$ is also Reedy cofibrant, 
and the inclusions $\sk_{n-1} B\hookrightarrow \sk_n B$ are Reedy cofibrations;
see e.g.\ \cite[Proposition~6.5]{BM}.

As shown in \cite[VII.3.6]{GJ}, the realization functor is left Quillen if
$\cK^{{\mathbf{\Delta}}^{\rm op}}$ is equipped with the Reedy model structure.
The following consequence is crucial.

\begin{lemma}
\label{invariance}
Let $\cK$ be a pointed simplicial model category and let $I$ be small.
\begin{itemize}
\item[{\rm (a)}]
If a diagram $X\colon I\to \cK$ is objectwise cofibrant, then $B_IX$ is Reedy cofibrant.
\item[{\rm (b)}]
If $f\colon X\to Y$ is an objectwise weak equivalence in $\cK^I$ and the diagrams $X$ and $Y$
are objectwise cofibrant, then the induced morphism $\hocolim_IX\to\hocolim_IY$ is a weak equivalence
of cofibrant objects.
\end{itemize}
\end{lemma}

\begin{proof}
Let $B=B_IX$. Thus $B_n=\coprod_{i_n\to\,\cdots\,\to i_0} X_{i_n}$ for all $n\ge 0$,
and we may write
\begin{equation}
\label{degenerate}
B_n=L_nB\textstyle\coprod Z_nB,
\end{equation}
where $L_nB$ includes the ``degenerate'' summands of~$B_n$, i.e., those labelled by 
sequences $i_n\to\cdots\to i_0$ where some arrow is an identity, and $Z_nB$ collects the rest.
Then the inclusion $L_nB\to B_n$ is a coproduct of the identity $L_nB\to L_nB$
and $*\to Z_nB$, which is a cofibration since $X$ takes cofibrant values. 
This proves part~(a).
Then part~(b) follows from the fact that realization is left Quillen,
since $B_If\colon B_IX\to B_IY$ is a weak equivalence between Reedy cofibrant objects.
\end{proof}

Thus, if defined as in (\ref{defhocolim}), 
the homotopy colimit is only homotopy invariant
on \emph{objectwise cofibrant} diagrams. 
For this reason, it is often convenient
to ``correct'' $\hocolim_I$ by composing it 
with a cofibrant replacement functor in~$\cK$, as in \cite[Definition~8.2]{Sh}.

The fundamental fact that, if made homotopy invariant, $\hocolim_I$
yields a total left derived functor of $\colim_I$ is explained as follows.
For each diagram $X\colon I\to\cK$ there is a natural morphism
\begin{equation}
\label{comparison}
\hocolim_I X\longrightarrow \colim_I X,
\end{equation}
since $\colim_I X$ is the coequalizer of the two face morphisms 
$\coprod_{i\to j} X_i \rightrightarrows \coprod_i X_i$;
that is, the morphism (\ref{comparison}) takes the form
\begin{equation}
\label{map}
X\otimes_I N(-\downarrow I)^{\rm op}_+ \longrightarrow X\otimes_I S.
\end{equation}
This morphism is a weak equivalence only in some cases; cf.\ \cite[XII.2]{BK}.
For instance, it is so if $I$ has a terminal object.
More importantly, (\ref{map}) is a weak equivalence if $X$ is \emph{cofibrant}
in the projective model structure of~$\cK^I$. To show this,
use the fact, proved in \cite[Theorem~3.2]{Ga}, that $(-)\otimes_I (-)$
is a left Quillen functor in two variables if the projective model structure
exists and is chosen on $\cK^I$ and the injective model structure
is considered in $\sSets_*^{I^{\rm op}}$. 
Accordingly, if $X$ is a projectively cofibrant diagram,
then $X\otimes_I(-)$ preserves weak equivalences between (objectwise) cofibrant objects,
so (\ref{map}) is indeed a weak equivalence.

It is also true, as shown in \cite[Theorem~3.3]{Ga}, that $(-)\otimes_I (-)$ is left 
Quillen in two variables if the projective model structure is
considered in $\sSets_*^{I^{\rm op}}$ and the injective model structure exists and is chosen on~$\cK^I$.
Thus, since $N(-\downarrow I)_+^{\rm op}\to S$ is a projectively cofibrant approximation
in $\sSets_*^{I^{\rm op}}$, the Bousfield--Kan formula 
displays in fact $\hocolim_I$ as a left derived functor of
$\colim_I$, provided that we restrict it to objectwise cofibrant diagrams
(i.e., cofibrant in the injective model structure).

For some purposes it is useful to consider the following functorial
projectively cofibrant replacement of a given diagram $X\colon I\to\cK$.
Assume that $X$ takes cofibrant values (or compose it with a cofibrant
replacement functor in $\cK$ otherwise).
Consider the functor $B_{(I\downarrow-)}X\colon I\times{\mathbf\Delta}^{\rm op}\to\cK$ given by
\[
(B_{(I\downarrow-)}X)(j,[n])=(B_{(I\downarrow j)}(X\circ U_j))_n=\coprod_{i_n\to\,\cdots\,\to i_0\to j} X_{i_n},
\]
where $U_j\colon (I\downarrow j)\to I$ sends each arrow $i\to j$ to~$i$,
and let $\widetilde X=|B_{(I\downarrow-)}X|$. Thus,
\begin{equation}
\label{resolution}
\widetilde X_j=|B_{(I\downarrow j)}(X\circ U_j)|=
\hocolim_{(I\downarrow j)}\,(X\circ U_j)
\end{equation}
for all $j\in I$.
Since $(I\downarrow j)$ has a terminal object for each~$j$, 
the natural morphism $\widetilde X\to X$ is an objectwise weak equivalence.
Using the fact that realization is a left adjoint and hence commutes with colimits,
one obtains a canonical isomorphism
\begin{multline}
\label{colimhocolim}
\colim_I\widetilde X=
\colim_j\widetilde X_j=\colim_j\,\hocolim_{(I\downarrow j)}(X\circ U_j)\\
=\colim_j|B_{(I\downarrow j)}(X\circ U_j)|
\cong|\colim_j B_{(I\downarrow j)}(X\circ U_j)|\cong |B_IX|=\hocolim_I X.
\end{multline}
In order to prove that the diagram $\widetilde X$ is indeed projectively cofibrant, 
view $B_{(I\downarrow-)}X$ as an object in $(\cK^I)^{{\mathbf\Delta}^{\rm op}}$
and check that it is Reedy cofibrant if the projective model structure is chosen in~$\cK^I$,
similarly as in part~(a) of Lemma~\ref{invariance}.

Although projectively cofibrant diagrams are not easy to characterize
in general, we note the following well-known special case for subsequent reference.

\begin{lemma}
\label{cofibrantsequences}
Let $\lambda$ be an infinite ordinal and let $\cK$ be a model category.
Suppose that, for an objectwise cofibrant diagram $X\colon \lambda\to\cK$,
each morphism $X_i\to X_{i+1}$ with $i<\lambda$ is a cofibration
and the induced morphism $\colim_{i<\alpha}X_i\to X_{\alpha}$
is also a cofibration for every limit ordinal $\alpha<\lambda$. 
Then the diagram $X$ is projectively cofibrant in~$\cK^{\lambda}$.
\end{lemma}

\begin{proof}
For each objectwise trivial fibration $A\to B$ in $\cK^{\lambda}$ and each morphism $X\to B$,
the existence of a lifting $X\to A$ follows by transfinite induction.
\end{proof}

For an objectwise cofibrant diagram $X\colon I\to \cK$, the homotopy colimit $\hocolim_I X$ can be filtered as follows.
Let $B=B_IX$, and denote $F_n=|\sk_n B|$. 
Since $B$ is Reedy cofibrant, the \emph{Bousfield--Kan map}
\[
\hocolim_{{\mathbf\Delta}^{\rm op}} B = B\otimes_{{\mathbf\Delta}^{\rm op}}N(-\downarrow{\mathbf\Delta}^{\rm op})^{\rm op}_+
\longrightarrow 
B\otimes_{{\mathbf\Delta}^{\rm op}}\Delta_+ = |B|
\]
is a weak equivalence; cf. \cite[XII.3.4]{BK}, \cite[18.7.1]{Hi}.
Since homotopy colimits commute,
\begin{align}
\label{filtration}
\hocolim_I X= |B| \simeq 
\hocolim_{{\mathbf\Delta}^{\rm op}} B \simeq &
\hocolim_{{\mathbf\Delta}^{\rm op}}\,\hocolim_n\sk_n B
\\
\notag
\cong & \hocolim_n\,\hocolim_{{\mathbf\Delta}^{\rm op}}\sk_n B\simeq\hocolim_n F_n.
\end{align}

This equivalence, which was our main goal in this subsection, 
is relevant in the context of triangulated categories, since
it allows us to replace a homotopy colimit indexed by an arbitrary small category by
another one indexed by a countably infinite ordinal, which fits into a well-known triangle
involving countable coproducts and the shift map, as in \cite[Definition~1.6.4]{N1}.

\subsection{Singly generated semilocalizing subcategories are coreflective}

The filtration displayed in (\ref{filtration}) of a homotopy colimit was used in \cite{B3}, \cite{B2}
to show that the class of acyclics of any homology theory on spectra is closed under homotopy colimits,
as a key ingredient of the proof of the existence of homological localizations.
The validity of the same argument for localizing subcategories of 
stable homotopy categories was suggested in \cite[Remark~2.2.5]{HPS}.
A~similar argument in derived categories of Grothendieck categories
was used for filtered homotopy colimits in~\cite[Theorem~3.1]{AJS2}.
We generalize it as follows.

\begin{propo}
\label{hocolim}
Let $\cK$ be a stable simplicial model category and let $\gamma\colon \cK\to\Ho(\cK)$
denote the canonical functor. Let $\cC$ be a semilocalizing subcategory of $\Ho(\cK)$.
If a diagram $X\colon I\to\cK$ is objectwise cofibrant and $\gamma X_i\in\cC$
for all $i\in I$, then $\gamma\hocolim_IX\in\cC$.
\end{propo}

\begin{proof}
Let $B=B_IX$ be the simplicial replacement of~$X$, as in (\ref{defhocolim}), and let
$F_n=|\sk_n B|$. As explained in \cite[VII.3.8]{GJ} or \cite[5.2]{GoS}, since realization
commutes with colimits, there is a natural pushout diagram
\begin{equation}
\label{latch}
\xymatrix{
(B_n\otimes\partial\Delta[n]_+)\coprod_{(L_nB\otimes\partial\Delta[n]_+)}(L_nB\otimes\Delta[n]_+) 
\ar[r] \ar[d] &
B_n\otimes\Delta[n]_+ \ar[d] \\
F_{n-1} \ar[r] & F_n \rlap{.}
}
\end{equation}

According to part~(a) of Lemma~\ref{invariance},
since the diagram $X$ is objectwise cofibrant, $B$ is Reedy cofibrant.
Hence, by Quillen's SM7 axiom for a simplicial model category
\cite[II.3.12]{GJ}, the upper arrow in (\ref{latch}) is a cofibration.
Therefore, the cofibre $F_n/F_{n-1}$ is isomorphic to the
cofibre of the upper arrow in~(\ref{latch}), which is isomorphic to
$Z_nB\otimes S^n$ if we write, as in~(\ref{degenerate}), $B_n=L_nB\coprod Z_nB$,
where $Z_nB$ contains the nondegenerate summands of~$B_n$. 
Hence, the sequence
\[
\xymatrix{
\gamma F_{n-1} \ar[r] & \gamma F_n \ar[r] & \Sigma^n \gamma Z_nB
}
\]
is part of a triangle in~$\Ho(\cK)$. Since $Z_nB$ is a coproduct
of objects $X_i$ with $i\in I$, it follows inductively that $\gamma F_n\in\cC$ for all~$n$.

Since $F_{n-1}\to F_n$ is a cofibration between cofibrant objects for every~$n$,
Lemma~\ref{cofibrantsequences} implies that
\[
\gamma\hocolim_n F_n\cong\gamma\colim_n F_n,
\]
and $\gamma\colim_n F_n$ is a cofibre of a morphism
$\coprod_n \gamma F_n\to\coprod_n \gamma F_n$ in~$\Ho(\cK)$, namely the difference between the identity
and the shift map, from which it follows that $\gamma\hocolim_n F_n$ is in~$\cC$, because
$\cC$ is semilocalizing.
Since, as observed in~(\ref{filtration}), $\hocolim_n F_n\simeq\hocolim_I X$, the claim is proved.
\end{proof}

\begin{coro}
\label{lambdafiltered}
If $\cK$ is a pointed simplicial combinatorial model category 
and we denote by $\gamma\colon\cK\to\Ho(\cK)$ the canonical functor, then
there is a regular cardinal $\lambda$ such that:
\begin{itemize}
\item[{\rm (a)}] 
For every $\lambda$\nobreakdash-filtered 
objectwise cofibrant diagram $X\colon I\to\cK$,
the natural morphism $\hocolim_I X\to\colim_I X$ is a weak equivalence.
\item[{\rm (b)}]
If $\cK$ is stable and $\cC$ is a semilocalizing subcategory of~$\Ho(\cK)$,
then, for every $\lambda$\nobreakdash-filtered diagram
$X\colon I\to\cK$ with $\gamma X_i\in\cC$ for all~$i$, we have
$\gamma\colim_IX\in\cC$.
\end{itemize}
\end{coro}

\begin{proof}
By \cite[Proposition~7.3]{Dug1}, for a combinatorial category $\cK$ there is a
regular cardinal $\lambda$ such that $\lambda$\nobreakdash-filtered colimits
of weak equivalences are weak equivalences.
Let $X\colon I\to\cK$ be an objectwise cofibrant diagram where $I$ is $\lambda$\nobreakdash-filtered.
Let $\widetilde X\to X$ be the objectwise weak equivalence defined in~(\ref{resolution}).
Then, by our choice of~$\lambda$, the induced morphism 
\[
\colim_I \widetilde X\longrightarrow\colim_I X
\]
is a weak equivalence. Since $\colim_I \widetilde X\cong\hocolim_I X$ by~(\ref{colimhocolim}),
part~(a) is proved.

Now let $X\colon I\to \cK$ be any diagram where $I$ is $\lambda$\nobreakdash-filtered,
and let $Q$ be a cofibrant replacement functor in~$\cK$. 
 From our choice of $\lambda$ we infer that $\colim_I X$ is weakly equivalent 
to $\colim_I QX$ and hence to $\hocolim_I QX$, by part~(a).
Therefore, if $\cK$ is stable, then for every semilocalizing subcategory $\cC$ of $\Ho(\cK)$ 
it follows from Proposition~\ref{hocolim} that $\gamma\colim_I X\in\cC$
if $\gamma X_i\in\cC$ for all~$i\in I$.
\end{proof}

We emphasize that the cardinal $\lambda$ in the statement of Corollary~\ref{lambdafiltered}
depends only on~$\cK$, not on the subcategory~$\cC$.

The following is another useful property of triangulated categories with models.
A~special case is discussed in~\cite[Remark~2.2.8]{HPS}. 
(The assumption that $\cK$ be simplicial is not really necessary
here nor in Corollary~\ref{lambdafiltered}, since 
homotopy colimits can be used, with the same basic properties, 
in all model categories; see \cite[Chapter~19]{Hi}.)

\begin{lemma}
\label{hocolims}
Let $\cK$ be a stable simplicial model category and denote by $\gamma\colon\cK\to\Ho(\cK)$ the canonical functor.
Let $I$ be any small category such that the projective model structure exists on~$\cK^I$.
Suppose given morphisms of objectwise cofibrant diagrams $X\to Y\to Z$ in $\cK^I$ such that
$\gamma X\to \gamma Y\to \gamma Z$ is part of a triangle in~$\Ho(\cK^I)$. Then
\[
\xymatrix{
\gamma\hocolim_{I} X \ar[r] & \gamma\hocolim_{I} Y
\ar[r] & \gamma\hocolim_{I} Z
}
\]
is part of a triangle in~$\Ho(\cK)$.
\end{lemma}

\begin{proof}
Since $\colim_I$ is left Quillen if the projective model structure is considered on~$\cK^I$,
its total left derived functor preserves triangles, as shown in \cite[Proposition~6.4.1]{Ho1}.

Alternatively, this result follows from the fact that homotopy colimits commute,
since, by assumption, $Z$ is weakly equivalent to the homotopy cofibre in $\cK^I$
of the given morphism $X\to Y$, i.e., the homotopy pushout of $*\leftarrow X\to Y$,
and $\hocolim_I$ is homotopy invariant on objectwise cofibrant diagrams.
\end{proof}

Special cases or variants of the next result have been described in
\cite[Theorem~3.4]{AJS1} for derived categories of Grothendieck categories; in 
\cite[Proposition~III.2.6]{BR} for compactly generated torsion pairs;
in \cite[Proposition~2.3.1]{HPS} for algebraic stable homotopy categories;
in \cite[Theorem~3.1]{KN} for derivators;
and in \cite[Proposition~16.1]{LuDAGI} for stable $\infty$\nobreakdash-categories.
The core of the argument was first used by Bousfield in~\cite{B3}. 

We note that, if the dual statement could be proved without large-cardinal assumptions, 
namely that singly generated colocalizing subcategories are reflective
in~ZFC, this would imply the existence of cohomological localizations of spectra in~ZFC,
a long-standing unsolved problem.

\begin{theo}
\label{singlygen}
If $\cK$ is a stable combinatorial model category,
then every singly generated semilocalizing subcategory of $\Ho(\cK)$
is coreflective.
\end{theo}

\begin{proof}
By Proposition~\ref{Dugger}, we may assume that $\cK$ is simplicial.
Let $\gamma\colon\cK\to\Ho(\cK)$ denote the canonical functor.
Let $\cC$ be a semilocalizing subcategory of $\Ho(\cK)$ and suppose that $\cC=\semiloc(A)$ for some object~$A$.
Pick, for each $n\ge 0$, a cofibrant object $B_n$ in~$\cK$ such that $\gamma B_n\cong\Sigma^nA$, and
choose a regular cardinal $\lambda$ and a fibrant replacement functor $R$ in $\cK$ such that:
\begin{itemize}
\item[(i)]
$B_n$ is $\lambda$\nobreakdash-presentable for every $n\ge 0$;
\item[(ii)]
all $\lambda$\nobreakdash-filtered colimits
of weak equivalences are weak equivalences;
\item[(iii)]
the functor $R$ preserves $\lambda$\nobreakdash-filtered colimits.
\end{itemize}
This is possible according to \cite[Proposition~2.3]{Dug1} and \cite[Proposition~7.3]{Dug1}, 
due to the assumption that $\cK$ is combinatorial.

In order to construct a coreflection onto~$\cC$,
we proceed similarly as in \cite[Proposition~1.5]{B3} or as in the proof of \cite[Proposition~2.3.17]{HPS}.
For any object $X$ of~$\cK$ ---which we assume fibrant and cofibrant---, take $Y_0=X$ and let 
$W_0$ be a coproduct of copies of $B_n$ for $n\ge 0$ indexed by all morphisms in $\cK(B_n,Y_0)$.
Let $u_0\colon W_0\to Y_0$ be given by $f\colon B_n\to Y_0$ on the summand corresponding to~$f$.

Next, let $Y_1$ be the homotopy cofibre of~$u_0$. More precisely,
factor $u_0$ into a cofibration $\widetilde u_0\colon W_0\to\widetilde Y_0$ followed by a trivial fibration
$\phi_0\colon\widetilde Y_0\to Y_0$;
let $Y_1'$ be the pushout of $\widetilde u_0$ and $W_0\to *$, 
and let $Y_1'\to Y_1''$ be a trivial cofibration with $Y_1''$ fibrant.
Since $Y_0$ is cofibrant, there is a left inverse $Y_0\to\widetilde Y_0$ to $\phi_0$
and hence a morphism $Y_0\to Y_1''$, which we factor again into a cofibration $v_0\colon Y_0\to Y_1$
followed by a trivial fibration $Y_1\to Y_1''$. Thus it follows from our choices that $Y_1$ is both 
fibrant and cofibrant, and 
\[
\xymatrix{
W_0\ar[r]^{u_0} & Y_0 \ar[r]^{v_0} & Y_1
}
\]
yields a triangle in~$\Ho(\cK)$, since 
\[
\xymatrix{
\gamma W_0\ar[r]^{\gamma u_0} & \gamma Y_0 \ar[r]^{\gamma v_0} & \gamma Y_1
}
\quad \mbox{is isomorphic to} \quad
\xymatrix{
\gamma W_0\ar[r]^{\gamma \widetilde u_0} & \gamma{\widetilde Y}_0 \ar[r] & \gamma Y_1',
}
\]
which is a canonical triangle.

Now repeat the process with $Y_1$ in the place of~$Y_0$.
In this way we construct inductively, for every ordinal~$i$, a sequence
\begin{equation}
\label{triangles}
\xymatrix{
W_i\ar[r]^{u_i} & Y_i \ar[r]^{v_i} & Y_{i+1}
}
\end{equation}
yielding a triangle in~$\Ho(\cK)$,
where $Y_{i+1}$ is fibrant and cofibrant, $v_i$ is a cofibration, and $W_i$ is a coproduct 
of copies of $B_n$ for $n\ge 0$, together with a morphism $w_{i+1}\colon X\to Y_{i+1}$ such
that $w_{i+1}=v_i\circ w_i$ (with $w_1=v_0$). 
If $\alpha$ is a limit ordinal, take $Z_{\alpha}=\colim_{i<{\alpha}}Y_i$,
and let $Z_{\alpha}\to Y_{\alpha}$ be a trivial cofibration with $Y_{\alpha}$ fibrant.
Since every (possibly transfinite) composition of cofibrations is a cofibration,
the morphism $X\to Z_{\alpha}$ given by $w_i$ for $i<{\alpha}$ is a cofibration, and hence
the composite $w_{\alpha}\colon X\to Y_{\alpha}$ is also a cofibration.

Let $Y\colon\lambda\to\cK$ be the diagram given by the objects $Y_i$ and the maps $v_i$ for $i<\lambda$.
Then $Y$ is cofibrant in $\cK^{\lambda}$, by Lemma~\ref{cofibrantsequences}, and
the constant diagram $X\colon\lambda\to\cK$ at the object $X$ is also cofibrant in~$\cK^{\lambda}$.
Let $F$ be the homotopy pullback of the map $X\to Y$ given by the morphisms $w_i$ 
and the trivial map $*\to Y$ in~$\cK^{\lambda}$.
Thus,
\[
\xymatrix{
\gamma F\ar[r] & \gamma X \ar[r] & \gamma Y
}
\]
is part of a triangle in $\Ho(\cK^{\lambda})$, since $\cK^{\lambda}$ is stable. 
Let $Q$ be a cofibrant replacement functor in $\cK$,
and let $QF$ be the composite of $Q$ and $F$. Thus $QF$ is objectwise cofibrant and,
by Lemma~\ref{hocolims},
\begin{equation}
\label{hotriangle}
\xymatrix{
\gamma\hocolim_{i<\alpha} QF_i\ar[r] & \gamma X \ar[r] & \gamma\hocolim_{i<\alpha} Y_i
}
\end{equation}
is part of a triangle for each limit ordinal $\alpha \le \lambda$.

By the octahedral axiom in~$\Ho(\cK)$ and (\ref{triangles}), there is a triangle 
\[
\xymatrix{
\gamma QF_i\ar[r] & \gamma QF_{i+1}\ar[r] & \gamma W_i
}
\] 
for each ordinal~$i$.
Therefore, it follows from transfinite induction that $\gamma QF_i\in\semiloc(A)$ for 
all ordinals~$i$, since each $\gamma W_i$ is constructed from $A$ by
means of suspensions and coproducts. If $\alpha$ is a limit ordinal,
then $\gamma Y_{\alpha}\cong\gamma\colim_{i<\alpha}Y_i\cong\gamma\hocolim_{i<\alpha}Y_i$
because $Y$ is cofibrant, and it then follows from (\ref{hotriangle}) that
$\gamma QF_{\alpha}\cong \gamma \hocolim_{i<\alpha} QF_i$, which is in $\semiloc(A)$
by Proposition~\ref{hocolim}.

Let $CX=\colim_{i<\lambda} QF_i$ and let $LX=\colim_{i<\lambda} Y_i$,
and note that the natural morphisms $\hocolim_{i<\lambda} QF_i\to CX$
and $\hocolim_{i<\lambda} Y_i\to LX$ are weak equivalences, by part~(a) of 
Corollary~\ref{lambdafiltered} and by our choice of~$\lambda$. By~Lemma~\ref{hocolims}, the sequence 
\[
\xymatrix{
\gamma CX\ar[r] & \gamma X\ar[r] & \gamma LX
}
\] 
is part of a triangle in~$\Ho(\cK)$. 
By~Proposition~\ref{hocolim}, $\gamma CX\in\semiloc(A)$.

Again by~our choice of~$\lambda$, we have
$RLX\cong \colim_{i<\lambda} RY_i$.
Now every $f\colon \Sigma^n A\to \gamma LX$ in $\Ho(\cK)$ 
can be lifted to a morphism $\widetilde f\colon B_n\to RLX$ in~$\cK$, as $RLX$ is fibrant.
Once more by our choice of~$\lambda$,
this morphism $\widetilde f$ factors through $RY_k$ for some $k<\lambda$, since $B_n$ is $\lambda$\nobreakdash-presentable.
Since~(\ref{triangles}) yields a triangle for all~$i$, the composite 
\[
\xymatrix{
\Sigma^n A\ar[r] & \gamma RY_k\ar[r] & \gamma RY_{k+1}
}
\] 
is zero.
This implies that $f\colon \Sigma^n A\to \gamma LX$ is zero.
Therefore, $\gamma LX\in A^{\,\rsemiperp}$, and, by~(\ref{closure}), $A^{\,\rsemiperp}=\semiloc(A)^{\rsemiperp}$.
This proves that $C$ is a coreflection onto~$\cC$, using Proposition~\ref{coreflecting}.
\end{proof}

We note that the reflection $L$ obtained in the previous proof is a nullification
$P_A$ in the sense of \cite{B1} and~\cite{DF}, and the subcategory $\cC$ is thus the closure under extensions of the
class of $A$\nobreakdash-cellular objects.

We have given the argument in full detail to stress the fact that it works for \emph{semilocalizing} subcategories.
It then also works for localizing subcategories, since, if $\cC$ is generated by an object $A$ as a localizing
subcategory, then it is generated by $\coprod_{n\le 0}\Sigma^nA$ as a semilocalizing subcategory.
However, in the case of a localizing subcategory, there is an alternative, much shorter proof of Theorem~\ref{singlygen}
which does not require the existence of models. Instead, it is based on Brown representability. 
A~similar argument can be found in \cite[Theorem~7.2.1]{Kr}.

\begin{propo}
\label{neeman}
Let $\cT$ be a well-generated triangulated category with coproducts. Then
every singly generated localizing subcategory of $\cT$ is coreflective.
\end{propo}
\begin{proof}
By \cite[Proposition 8.4.2]{N1}, the category $\cT$ satisfies
Brown representability, and $\cT=\cup_{\alpha}T^{\alpha}$, i.e., every object of $\cT$ is 
$\alpha$-compact for some infinite cardinal~$\alpha$.

Let $\cC$ be a localizing subcategory of $\cT$ generated by some object~$A$. 
Then $A\in\cT^{\alpha}$ for some infinite cardinal~$\alpha$. 
Hence, it follows from \cite[Corollary 4.4.3]{N1} that the Verdier quotient 
category $\cT/\cC$ has small hom-sets. 

The existence of a coreflection onto $\cC$ amounts to the existence of a right adjoint to the inclusion 
$\cC\hookrightarrow \cT$, and this is equivalent to the existence of a right adjoint to the functor 
$F\colon \cT\to \cT/\cC$ (see \cite[Proposition 9.1.18]{N1}). 
Since $\cT/\cC$ has small hom-sets, a right adjoint $G\colon \cT/\cC\to \cT$ can be defined
as follows. If $X$ is any object of $\cT/\cC$, then
$GX$ is obtained by Brown representability, namely $(\cT/\cC)(F(-),X)\cong \cT(-,GX)$.
\end{proof}

Recall from \cite[Proposition~6.10]{R} that, if $\cK$ is a stable combinatorial model category,
then $\Ho(\cK)$ is indeed well generated.

\subsection{Semilocalizing subcategories are singly generated}

We remark that
the way in which Vop\v{e}nka's principle is used in Theorem~\ref{mainthm2} below
is different from the way in which it was used in Section~\ref{RCS}.
What we need here is the fact that, by \cite[Theorem~6.6 and Corollary~6.18]{AR},
\emph{if Vop\v{e}nka's principle holds, then every full subcategory of a locally presentable
category closed under $\lambda$\nobreakdash-filtered colimits for some regular cardinal $\lambda$ is accessible.} 
The following argument was used similarly in \cite[Lemma~1.3]{CCh} and \cite[Lemma~1.3]{Cho}.

\begin{theo}
\label{mainthm2}
Let $\cK$ be a stable combinatorial model category.
If Vop\v{e}nka's principle holds, then every semilocalizing subcategory of $\Ho(\cK)$
is singly generated and coreflective.
\end{theo}

\begin{proof}
First replace $\cK$ with a Quillen equivalent stable simplicial combinatorial
model category, which is possible according to Proposition~\ref{Dugger}.
Let $\cC$ be a semilocalizing subcategory of $\cT=\Ho(\cK)$.
Write it as the union of an ascending chain of full subcategories
\[ \cC=\bigcup_{i\in\Ord} \cC_i, \]
indexed by the ordinals,
where for each $i$ there is an object $A_i\in\cC$ such that $\cC_i=\semiloc(A_i)$.
Then, by Theorem~\ref{singlygen}, each $\cC_i$ is coreflective.

Consider the corresponding classes $\cS_i=\gamma^{-1}(\cC_i)$,
where $\gamma\colon\cK\to\Ho(\cK)$ is the canonical functor.
These form an ascending chain of full subcategories of~$\cK$.
Let $\cS=\cup_{i\in\Ord}\, \cS_i=\gamma^{-1}(\cC)$. By Corollary~\ref{lambdafiltered}, 
there is a regular cardinal $\lambda$ such that
each $\cS_i$ is closed under $\lambda$-filtered colimits, and so is~$\cS$.

Since $\cK$ is locally presentable, 
Vop\v{e}nka's principle implies that $\mathcal S$ is accessible \cite[Theorem~6.6 and Corollary~6.18]{AR}.
Hence, there is a regular cardinal~$\mu$, which we may choose bigger than~$\lambda$, 
and a set $\mathcal X$ of $\mu$\nobreakdash-presentable objects 
in $\mathcal S$ such that every object of $\mathcal S$ is a $\mu$\nobreakdash-filtered
colimit of objects from~$\mathcal X$.

Since $\mathcal X$ is a set, we have ${\mathcal X}\subseteq \cS_k$ for some ordinal~$k$.
Hence, every object of $\mathcal S$ is a $\mu$\nobreakdash-filter\-ed colimit of objects from~$\cS_k$.
But the class $\cS_k$ is closed under $\mu$\nobreakdash-filtered colimits, since every $\mu$\nobreakdash-filtered colimit
is also $\lambda$\nobreakdash-filtered. Therefore, $\cS_k=\cS$,
that is, the chain $\{\cS_i \mid i\in\Ord\}$ eventually stabilizes.
Then $\{ \cC_i\mid i\in\Ord\}$ also stabilizes, since $\cC_i = \gamma(\cS_i)$ for all~$i$. 
This proves that $\cC=\cC_k$ for some~$k$, which is singly generated and coreflective.
\end{proof}

Under the assumptions of Theorem~\ref{mainthm2}, every localizing subcategory $\cC$
is also singly generated, since we may infer from Theorem~\ref{mainthm2} that 
$\cC=\semiloc(A)$ for some object~$A$, and then $\cC=\loc(A)$ as well.

It also follows that, under the assumptions of Theorem~\ref{mainthm2}, all semilocalizing subcategories
(and all localizing subcategories) are closed.
If we assume, in addition, that $\cT$ is tensor triangulated,
and apply Theorem~\ref{mainthm2} to a localizing ideal, then the corresponding
coreflection $C$ is a colocalization in the sense of~\cite{HPS}; that is, if $X$ is such that $CX=0$,
then $C(F(E,X))=0$ for every object $E$ in~$\cT$.

Hence, the question asked after \cite[Lemma~3.6.4]{HPS} of whether all localizing ideals are closed
has an affirmative answer in tensor triangulated
categories with combinatorial models, assuming Vop\v{e}nka's principle.

\section{Nullity classes and cohomological Bousfield classes}
\label{BC}

It follows from Theorem~\ref{ref-coref}, Theorem~\ref{mainthm1} and Theorem~\ref{mainthm2}
that, if Vop\v{e}nka's principle holds, then
in every triangulated category $\cT$ with combinatorial models there is a
bijective correspondence between localizing subcategories and
colocalizing subcategories. This answers affirmatively \cite[Problem~7.3]{N2}
under the assumptions made here.

In fact, under the same assumptions, there is also a bijective correspondence between
semilocalizing subcategories and semicolocalizing subcategories. Hence, we have:

\begin{coro}
Under Vop\v{e}nka's principle, every semilocalizing subcategory of a triangulated category
with combinatorial models is part of a $t$\nobreakdash-structure, 
and the same happens for every semicolocalizing subcategory.
\end{coro}

\begin{proof}
As stated in~Theorem~\ref{superbijection}, every reflective semicolocalizing subcategory
yields a $t$\nobreakdash-structure, and so does every coreflective semilocalizing subcategory.
Theorem~\ref{mainthm1} ensures reflectivity of all semicolocalizing subcategories
and Theorem~\ref{mainthm2} ensures coreflectivity of all semilocalizing subcategories,
under the assumptions made.
\end{proof}

Another consequence of our results is the following.

\begin{theo}
\label{perpperp}
Let $\cT$ be a triangulated category with combinatorial models.
Assuming Vop\v{e}nka's principle,
every semicolocalizing subcategory of $\cT$ 
is equal to $E^{\,\rsemiperp}$ for some object~$E$ and every colocalizing subcategory 
is equal to $E^{\boldperp}$ for some~$E$.
\end{theo}

\begin{proof}
Let $\cL$ be a semicolocalizing subcategory of~$\cT$. 
Theorem~\ref{mainthm1} ensures that $\cL$ is reflective and
hence $\cL=({}^{\lsemiperp}\cL)^{\rsemiperp}$,
by Proposition~\ref{cid}.
Now consider ${}^{\lsemiperp}\cL$, which is a semilocalizing subcategory,
hence singly generated by Theorem~\ref{mainthm2}. That is, ${}^{\lsemiperp}\cL=\semiloc(E)$ for some~$E$.
Consequently, $\cL=({}^{\lsemiperp}\cL)^{\rsemiperp}=\semiloc(E)^{\rsemiperp}= E^{\,\rsemiperp}$
by~(\ref{closure}),
which proves our first claim. We argue in the same way for a colocalizing subcategory.
\end{proof}

Semicolocalizing subcategories of the form $E^{\,\rsemiperp}$ for some object $E$ are called \emph{nullity classes},
since $E^{\,\rsemiperp}$ consists of objects $X$ that are \emph{$E$\nobreakdash-null}, in the sense that
$\cT(\Sigma^k E,X)=0$ for $k\ge 0$ (this terminology is consistent with \cite{Cha} or~\cite{DF}, but slightly
differs from that used in~\cite{Sta}). Thus, the following corollary
is a rewording of Theorem~\ref{perpperp}.

\begin{coro}
Assuming Vop\v{e}nka's principle, every semicolocalizing subcategory of a 
triangulated category with combinatorial models is a nullity class.
\label{coho}
\end{coro}

It was shown in \cite{Sta} that there is a proper class of distinct nullity classes
$E^{\,\rsemiperp}$ in the derived category of $\Z$ or in the homotopy category of spectra. 
However, it is unknown if there is a proper class or only a set of distinct classes of the form~$E^{\boldperp}$.

The same problem is open for classes of the form ${}^{\boldperp}E$.
A localizing subcategory
of the form ${}^{\boldperp}E$ for some object $E$ is called a \emph{cohomological Bousfield class}; cf.~\cite{Ho2}.
It follows from
Corollary~\ref{closed_implies_coreflective} that cohomological Bousfield classes of
spectra are coreflective under Vop\v{e}nka's principle ---this was first proved in \cite{CCh}, \cite{CSS}.
However, we do not know if every localizing subcategory of spectra is a cohomological Bousfield class.
Indeed, we could not prove that colocalizing subcategories
are singly generated, not even under Vop\v{e}nka's principle and in the presence of
combinatorial models. As we next explain, there seems to be a reason for this.

\subsection{Torsion theories in abelian categories}
In an abelian category, the analogue of a semilocalizing subcategory
is a full subcategory closed under colimits and extensions (this is usually called
a \emph{torsion class}), and the analogue of a
semicolocalizing subcategory is a full subcategory closed under limits and
extensions (called a \emph{torsion-free class}).
In well-powered abelian categories, 
torsion classes are coreflective and torsion-free classes are reflective;
see~\cite{Di}. 

A torsion class closed under subgroups is called \emph{hereditary}.
These correspond to the localizing subcategories.
Hereditary torsion classes of modules over a ring are singly generated
and their orthogonal torsion-free classes are also singly generated; see~\cite{DH2}.
In the non-hereditary case, the situation is more intriguing. 
On one hand, under Vop\v{e}nka's principle, every torsion
class of abelian groups is singly generated. This was shown in \cite{DH1} and~\cite{GS}
(we note that the proof of Theorem~\ref{mainthm2} can easily be adapted so as to hold
for abelian groups, thus yielding another proof of this fact).
On the other hand, there are torsion-free classes that are not singly generated
in~ZFC; for example, the class of abelian groups whose countable subgroups are free
---see~\cite[Theorem~5.4]{DG}. 
This casts doubt on the fact that, in reasonably restricted triangulated
categories, colocalizing or semicolocalizing subcategories are necessarily singly generated, 
even under large-cardinal assumptions.

\end{document}